\documentclass{amsart}
\usepackage{color}

\usepackage{amsmath} 
\usepackage{amssymb} 
\usepackage{mathrsfs}
\usepackage{mathtools}

\newtheorem{theorem}{Theorem}[section] 
\newtheorem{claim}[theorem]{Claim}

\newtheorem{observation}[theorem]{Observation}

\theoremstyle{definition}
\newtheorem{definition}[theorem]{Definition}
\newtheorem{example}[theorem]{Example}
\newtheorem{problem}[theorem]{Problem}
\newtheorem{discussion}[theorem]{Discussion}

\newtheorem{hypothesis}[theorem]{Hypothesis}

\theoremstyle{remark}
\newtheorem{remark}[theorem]{Remark}

\newtheorem{conclusion}[theorem]{Conclusion}

\DeclareMathOperator{\hL}{hL}
\DeclareMathOperator{\hd}{hd}

\newcommand{\rest}{{\restriction}}
\newcommand{\dom}{{\rm dom}} 

\newcommand{\ap}{{\rm ap}}

\newcommand{\supp}{{\rm supp}}
\newcommand{\wilog}{{\rm without loss of generality}}

\newcommand{\blueq}[1]{{\color{blue} #1}}

\newcommand{\then}{{\underline{then}}}
\newcommand{\when}{{\underline{when}}}
\newcommand{\Then}{{\underline{Then}}}

\newcommand{\cH}{{\mathcal H}}
\newcommand{\cJ}{{\mathcal J}}

\newcommand{\cI}{{\mathcal I}}

\newcommand{\bbL}{{\mathbb L}}

\newcommand{\bbP}{{\mathbb P}}

\newcommand{\bbQ}{{\mathbb Q}}

\newcommand{\cU}{{\mathcal U}}

\newcommand{\cf}{{\rm cf}}
\newcommand{\pr}{{\rm pr}}

\newcount\skewfactor
\def\mathunderaccent#1#2 {\let\theaccent#1\skewfactor#2
\mathpalette\putaccentunder}
\def\putaccentunder#1#2{\oalign{$#1#2$\crcr\hidewidth
\vbox to.2ex{\hbox{$#1\skew\skewfactor\theaccent{}$}\vss}\hidewidth}}
\def\name{\mathunderaccent\tilde-3 }

\newenvironment{PROOF}[2][\proofname.]
   {\begin{proof}[#1]\renewcommand{\qedsymbol}{\textsquare$_{\rm #2}$}}
   {\end{proof}}

\usepackage[hidelinks]{hyperref}

\begin{document}
\makeatletter\def\shfiuwefootnote{\gdef\@thefnmark{}\@footnotetext}\makeatother\shfiuwefootnote{Version 2023-08-16. See \url{https://shelah.logic.at/papers/918/} for possible updates.}

\title {Many partition relations below density\\
 Sh918}
\author {Saharon Shelah}
\address{Einstein Institute of Mathematics\\
Edmond J. Safra Campus, Givat Ram\\
The Hebrew University of Jerusalem\\
Jerusalem, 91904, Israel\\
 and \\
 Department of Mathematics\\
 Hill Center - Busch Campus \\ 
 Rutgers, The State University of New Jersey \\
 110 Frelinghuysen Road \\
 Piscataway, NJ 08854-8019 USA}
\email{shelah@math.huji.ac.il}
\urladdr{http://shelah.logic.at}
\thanks{The author thanks Alice Leonhardt for the beautiful typing for the journal version up to 2019. Research supported by the United States-Israel Binational Science Foundation (Grant No. 2002323). For the changes after 2019, the author would like to thank the typist for his work and is also grateful for the generous funding of typing services donated by a person who wishes to remain anonymous. First Typed - 06/Dec/21.}

 %Previous version - 2011/August/18
 %formerly F821

\subjclass[2010]{Primary 03E35, 03E02; Secondary: 54A25, 54A35}

\keywords{set theory, independence, forcing, partition relations, topological cardinal invariants, hereditary density, hereditary Lindelof}

\date{August 16, 2023}

\begin{abstract}
    We force $2^\lambda$ to be large and for many pairs in the
    interval $(\lambda,2^\lambda)$ a strong version of the polarized
    partition relations hold.  We apply this to problems in general
    topology.  E.g. consistently, every $2^\lambda$ is successor of
    singular and for every Hausdorff regular space $X$, $\hd(X) \le
    s(X)^{+3}$, {\rm hL}$(X) \le s(X)^{+3}$ and better when $s(X)$ is 
    regular, via a half-graph partition relations.  
    For the case $s(X) = \aleph_0$ we get
    $\hd(X)$, $\hL(X) \le \aleph_2$.

    \noindent
    Minor changes in July, 2020. 
\end{abstract}

\maketitle

\numberwithin{equation}{section}

\setcounter{section}{-1}

%%%%%%%%%%%%%%%%%%%%%%%%%%%%%%%%%%%%%%%%%%%
% Annotated Content
\newpage

\section*{Annotated Content}

\S0 \quad Introduction, pg. \pageref{0}. 

\S1 \quad A Criterion for Strong Polarized Partition Relations, pg. \pageref{1}.

\begin{enumerate}
    \item[${{}}$]   [We give sufficient conditions for having strong
    versions of polarized partition relations after forcing.]
\end{enumerate}

\S2 \quad The forcing, pg. \pageref{2}.

\begin{enumerate}
    \item[${{}}$]   [Assume GCH for simplicity and $\mathbf p$ a parameter
    with  $\lambda < \mu$ regular and $\Theta \subseteq \text{
    Reg } \cap [\lambda,\mu^+)$ and we define $\bbQ_{\mathbf p}$ which adds
    $\mu$ Cohen subsets to $\lambda$ but have many kinds of supports, one
    for each $\theta \in \Theta$, influencing the order.]
\end{enumerate}

\S3 \quad Applying the criterion, pg. \pageref{3}.

\begin{enumerate}
    \item[${{}}$]   [The main result is that (cardinal arithmetic is
    changed just by making $2^\lambda = \mu$ and) using \S1 we prove the
    strong version of polarized partition relations hold in many instances.]
\end{enumerate}

%%%%%%%%%%%%%%%%%%%%%%%%%%%%%%%%%%%%%%%%%%%%%
% Section 0
\newpage

\section{Introduction}\label{0} 

Out motivation is a problem in general topology and for this we get a consistency result in the partition calculus.

In Juhasz-Shelah \cite{Sh:899} was proved: if $(\forall \mu < \lambda)(\mu^{\aleph_0} < \lambda)$ \then \, there is a c.c.c. forcing notion that adds a regular topological space,
hereditarily Lindel\"of of density $\lambda$.

A natural question asked there (\cite{Sh:899}) is:

\begin{problem}\label{0z.1}  
    Assume $\aleph_1 < \lambda \le 2^{\aleph_0}$. Does there exist (i.e., provably in ZFC) a hereditary  Lindel\"of regular space of density $\lambda$?
    
    On cardinal invariants in general topology see \cite{Ju80}.
    
    We prove the consistency of a negative answer, in fact of stronger results by
    proving the consistency of strong variants of polarized partition
    relations (the half-graphs, see below).  They are strong  enough to resolve the question about hereditary density (and hereditary Lindel\"of).  
    Moreover, if $\lambda = \lambda^{<\lambda} <
    \mu = \mu^{< \mu}$ (and G.C.H. holds in $[\lambda,\mu)$), then there is a
    forcing extension making $2^\lambda \ge \mu$ neither  adding new $(< \lambda)$-sequences nor
    collapsing cardinals such that for many pairs $\lambda_* < \mu_*$ in
    the interval we have the appropriate partition relations. 
    
    An earlier result is in the paper \cite[Theorem 1.1, pg.357]{Sh:276}
    and it states the following: if $\lambda > \kappa > \mu$ are regular
    cardinals, $\lambda > \kappa^{++}$, then there is a cardinal and
    cofinality preserving forcing that makes $2^\mu = \lambda$ and
    $\kappa^{++} \rightarrow (\kappa^{++},(\kappa;\kappa)_\kappa)^2$ in
    addition to the main result there $2^\lambda \rightarrow
    [\lambda]^2_3$, see more in \cite{Sh:289}, \cite{Sh:288}, \cite{Sh:481}, \cite{Sh:546}.  The applied notion of forcing $(Q,\le)$
    is the following: $p \in Q$ if $p$ is a function from a subset $\dom(p)
    \in [\lambda]^{\le \kappa}$ into Add$(\mu,1) - \{\emptyset\}$ where
    Add$(\mu,1)$ denotes the forcing adding a Cohen subset of $\mu$.  $p
    \le q$ if $\dom(p) \supseteq \dom(q), p(\alpha) \le q(\alpha)$
    for $\alpha \in \dom(q)$ and $|\{\alpha \in \dom(q):p(\alpha) \ne q(\alpha)\}| < \mu$.
    
    For simultaneously many $n$-place polarized partition relation  Shelah-Stanley \cite{Sh:608} deals with it but there are problems there, so we do not rely on it.
    
    Our main result in general topology is 
    Theorem \ref{1t.41}, by it: consistently, G.C.H. fails badly ($2^\mu$ is a successor of a limit cardinal $> \mu$ except when $\mu$ is strong limit singular and then $2^\mu = \mu^+$) and $\hd(X)$, $\hL(X)$ are $\le s(X)^{+3}$ for every Hausdorff regular $X$ and $|X| \le 2^{(\hd(X))^+},w(X) \le 2^{(\text{hL}(X))^+}$ for any Hausdorff $X$.  (Usually $s(X)^{+2}$ suffice so in particular  ``$X$ is hereditary Lindel\"of $\Rightarrow X$ has density 
    $\le \aleph_2$".
    
    Concerning partition relations we give a generalization of the earlier result explained above, namely, the consistency of $2^{\aleph_0} = \lambda$ and $\mu^{++} \rightarrow (\mu,(\mu;\mu)_\mu)^2$ simultaneously holding for each regular cardinal $\mu$ such that $\mu^{++} \le \lambda$.  This gives a
    model in which though GCH fails badly, we have strong enough partition relations implying that the hereditary density and the hereditary Lindel\"of numbers of a $T_3$ space $X$ are bounded by
    $s(X)^{+3}$ where $s(X)$ stands for spread.
    
    The notion of forcing $(P,\le)$ used for the argument is defined as
    follows.  For each regular cardinal $\mu < \lambda$ define the
    following equivalence relation $E_\mu$ on $\lambda$.  $xE_\mu y$ iff
    $x + \mu = y + \mu$.  Let $[x]_\mu$ denote the equivalence class of
    $x$.  $p \in P$ if $p$ is a function from some set $\dom(p) \subseteq
    \lambda$ into $\{0,1\}$ such that $|[x]_\mu \cap \dom(p)| <
    \mu$ holds for every successor $\mu < \lambda,x < \lambda$.  $p \le q$
    if $p \supseteq q$ and for every successor $\mu < \lambda$ we have:
    \[
    |\{[x]_\mu:\emptyset \ne \dom(q) \cap [x]_\mu \ne \text{Dom}(p) \cap [x]_\mu\}| < \mu.
    \]
    
    This notion of forcing $(P,\le)$, in a most remarkable way, imitates concurrently several different posets $(Q,\le)$ as defined above.  Not surprisingly, in order to show that $(P, \le)$ is cardinal and cofinality preserving, the author uses ideas similar to those in \cite{Sh:276}.
    
    In order to prove the main claim, that is, the partition relation, we use the following trick: we find a condition $\bar p$ such that the dense sets we are interested in are all dense below $\bar p$.  It suffices, therefore, to show that forcing with the part below $\bar p$ gives the required result, and this reduces the problem to showing that a certain notion of forcing $(R,\le)$ forces the sought-for-partition relation where $|R|$ is small (compared to $\mu$).  As $(R,<)$ is close to the poset $(Q,<)$ of \cite{Sh:276}, an elementary sub-model argument similar to the one there applies.
    
    The exposition of the method is axiomatic; the author formulates the most general situation where this method works, and then specifies it to the situation sketched above. This is not necessarily the optimal description for those who are only interested in the application given.  There is, however, reason for the peculiar way of presenting
    this proof: we would like to include this method into the tool kit set, and simply quote it at possible later applications.
\end{problem}

Recall (first appeared in Erd\"os-Hajnal \cite{ErHa78}, but probably raised by Galvin in letters in the mid seventies):

\begin{definition}\label{0z.4}  
    1) $\lambda \rightarrow (\mu;\mu)^2_\kappa$ means that: 
    
    for every $\mathbf c:[\lambda]^2
    \rightarrow \kappa$ there are $\varepsilon$ and $\alpha_i,\beta_i$ for $i < \mu$ such that:
    
    \begin{enumerate}
        \item[$(a)$]   $\varepsilon < \kappa,$
        
        \item[$(b)$]   if $i < j < \mu$ then $\alpha_i < \beta_i < \alpha_j < \lambda,$
        
        \item[$(c)$]   if $i \le j < \mu$ then $\mathbf c\{\alpha_i,\beta_j\} =
        \varepsilon$.
    \end{enumerate}
    
    2) We can replace $\mu$ by an ordinal and if $\kappa=2$ we may omit it.
\end{definition}

\begin{definition}\label{0.3}
    1) Let $\lambda \rightarrow (\mu,(\mu;\mu)_\kappa)^2$ means that:
    
    for every $\mathbf c:[\lambda]^2 \rightarrow 1 + \kappa$ there are 
    $\varepsilon$ and $\alpha_i,\beta_i$ for $i < \mu$ such that:
    
    \begin{enumerate}
        \item[$(a)$]   $\varepsilon < \kappa,$
        
        \item[$(b)$]   $\alpha_i < \beta_i < \alpha_j < \lambda$ for $i < j < \mu,$
        
        \item[$(c)_0$]   if $\varepsilon = 0$ then $i < j \Rightarrow \mathbf
        c\{\alpha_i,\alpha_j\} = \varepsilon$, so we can forget the $\beta_i$'s, 
        
        \item[$(c)_1$]   if $\varepsilon \ge 1$ then $i \le j \Rightarrow
        \mathbf c\{\alpha_i,\beta_j\} = \varepsilon$.
    \end{enumerate}
    
    2) In part (1) if $\kappa =1$ we may omit it.  Above replacing $\mu$ by
    ``$< \mu$" means ``for every $\xi < \mu$ we have ....".
    
    We thank Shimoni Garti for many corrections and Istvan Juhasz for  questions and historical remarks; we may continue this research in 
    \cite{Sh:F884}.
\end{definition}

%%%%%%%%%%%%%%%%%%%%%%%%%%%%%%%%%%%%%%%%%%%%%%%%%%%%%%
% Section 1
\newpage

\section{Strong polarized partition relations}\label{1} 

We deal with sufficient conditions on a forcing notion for preserving such partition relations.  For this, we use an expansion of a forcing notion.  Instead of the usual pair $(Q,\le_{\mathbf Q})$, namely,
the underlying set and the partial order, we use a quadruple of the form $\mathbf Q = (Q,\le_{\mathbf Q},\le^{\pr}_{\mathbf Q}$, ap$_{\mathbf Q})$.

The ``pr" stands for pure, and the ``ap" stands for apure.  Both are included (as partial orders) in $\mathbf Q$.

\begin{discussion}
    We define (below) the notion of 
    $``(\lambda,\theta,\xi)$-forcing" to give a sufficient condition for
    appropriate cases of the partition relations defined above to hold.
    We start with the quadruple $\mathbf Q = (Q,\le_{\mathbf Q},\le^{\pr}_{\mathbf
    Q},\ap_{\mathbf Q})$ such that $q \in Q \Rightarrow 
    \ap_{\mathbf Q}(q) \subseteq Q$ and $\le_{\mathbf
    Q},\le^{\pr}_{\mathbf Q}$ are quasi orders on $Q$.  The
    idea is that if $r \in \ap_{\mathbf Q}(q)$ then $r$ and $q$ are
    compatible in $\bbQ$, close to ``$r$ is an a-pure extension of $q$".
\end{discussion}

\begin{definition}
\label{1c.15}  
1) We say that $\mathbf Q$ is a 
$(\chi^+,\theta,\xi)$-forcing notion \when \, $\chi^+,\theta$ are
regular uncountable cardinals, $\xi$ an ordinal and $\circledast$
below holds; in writing $(\chi^+,\theta,< \zeta)$ we mean that 
$\circledast$ holds for every $\xi < \zeta$; also we can replace
$\chi^+$ by $\lambda$:

\begin{enumerate}
\item[$\circledast$]   $(a) \quad \mathbf Q = (Q,\le_{\mathbf Q},
\le^{\pr}_{\mathbf Q},\ap_{\mathbf Q}),$

\item[${{}}$]   $(b) \quad \mathbb Q = (Q,\le_{\mathbf Q})$ is a forcing
notion (i.e. a quasi order, so $\Vdash_{\mathbf Q}$ means 
$\Vdash_{\bbQ},$ 

\hskip25pt and $p \in \mathbf Q$ means 
$p \in Q$ and $\mathbf V^{\mathbf Q}$ means $\mathbf V^{\bbQ}$ and

\hskip25pt $\name{\mathbf G}$ is the $\bbQ$-name of the generic set), 

\item[${{}}$]   $(c) \quad \le^{\pr}_{\mathbf Q}$ is a quasi order on $Q$
and $p \le^{\pr}_{\mathbf Q} q$ implies $p \le_{\mathbf Q} q,$

\item[${{}}$]   $(d)(\alpha) \quad$ ap$_{\mathbf Q}$ is a function
with domain $Q,$

\item[${{}}$]  $\quad (\beta) \quad$ for $q \in {\mathbf Q}$
we have\footnote{it is natural to demand $q \in \ap_{\mathbf Q}(q)$,
but not really necessary (if we do not demand it, this 
just complicates a little $\circledast(C)(d)$).} 
$q \in \ap_{\mathbf Q}(q) \subseteq Q,$ 

\item[${{}}$]   $\quad (\gamma) \quad r \in \ap_{\mathbf Q}(q) 
\Rightarrow r,q$ are compatible in $\mathbb Q$; moreover,

\item[${{}}$]   $\quad (\gamma)^+ \quad$ if $r \in \ap_{\mathbf
Q}(q) \wedge q \le^{\pr}_{\mathbf Q} q^+$ then $q^+,r$ are
compatible in $\bbQ$

\hskip25pt  moreover there is $r^+ \in \ap_{\mathbf Q}(q^+)$
such that $q^+ \Vdash_{\mathbf Q} ``r^+ \in \name{\mathbf G}_{\mathbf Q} 
\Rightarrow$

\hskip25pt $r \in \name{\mathbf G}_{\mathbf Q}"$\footnote{no harm in 
asking that $r \le^{\pr}_{\mathbf Q} s$ and 
$s \in \ap_{\mathbf Q}(q^+)$ and $q^+ \le s$ for some $s$.  Why
this does not follow from our assumption?  
By the present demand $r^+,q^+$ have a common $\le$-upper bound which
is $s$, so $s \Vdash ``q^+,r^+ \in \name{\mathbf G}_{\mathbf Q}$
hence $r \in \name{\mathbf G}_{\mathbf Q}"$ so \wilog \, $r \le s$,
\underline{but} this does not say $q \le^{\pr}_{\mathbf Q} s$.}, 

\item[${{}}$]   $(e) \quad (Q,\le^{\pr}_{\mathbf Q})$ is $(<
\theta)$-complete, i.e. any $\le^{\pr}_{\mathbf Q}$-increasing
sequence of length

\hskip25pt  $< \theta$ has a $\le^{\pr}_{\mathbf Q}$-upper bound
is $Q,$

\item[${{}}$]  $(f) \quad (Q,\le^{\pr}_{\mathbf Q})$ satisfies the
$\chi^+$-c.c.,

\item[${{}}$]   $(g) \quad$ if $\bar q = \langle q_\varepsilon \colon \varepsilon <
\theta\rangle$ is $\le^{\pr}_{\mathbf Q}$-increasing \then \,
\footnote{Note that: we can restrict ourselves to the case $q_0 \in
{\cI}$, where ${\cI}$ is a dense subset of $\bbQ$.  Also we can restrict
ourselves to the set of $\bar q$ sequences which is the set of plays of a
 suitable game with one player using a fixed strategy, etc.} for
stationary many limit 

\hskip25pt ordinals $\zeta < \theta$, the sequence $\bar q
\restriction \zeta$ has an exact $\le^{\pr}_{\mathbf Q}$-upper bound,

\hskip25pt see part (2) below, 

\item[${{}}$]   $(h) \quad$ if $\langle q_\varepsilon \colon \varepsilon <
\theta\rangle$ is $\le^{\pr}_{\mathbf Q}$-increasing and
$p_\varepsilon \in \ap_{\mathbf Q}(q_\varepsilon)$ for

\hskip25pt  $\varepsilon < \theta$ and $\xi < \theta$ 
\then \, for some $\zeta < \theta$ 
we have $q_\zeta
\Vdash_{\mathbf Q}$ ``if $p_\zeta \in \name {\mathbf G}_{\mathbf Q}$ 

\hskip25pt  then $\xi \le \rm{otp}\{\varepsilon <
\zeta \colon p_\varepsilon \in \name{\mathbf G}_{\mathbf Q}\}"$,

\item[${{}}$]   $(i) \quad$ if $q \in \mathbf Q$ \then \,
$\ap_{\mathbf Q}(q)$ has cardinality $< \theta$, 

\item[${{}}$]   $(j) \quad$ if $q_* \le r$ \then \, there
is a $(q_*,r)$-witness $(q,p)$ which means:

\begin{enumerate}
    \item[${{}}$]  $\bullet_1 \quad q_* \le^{\pr}_{\mathbf Q} q,$
    
    \item[${{}}$]  $\bullet_2 \quad p \in \ap_{\mathbf Q} (q_*),$
    
    \item[${{}}$]   $\bullet_3 \quad q \Vdash_{\mathbb Q} ``p \in \name{\mathbf
    G} \Rightarrow r \in \name{\mathbf G}"$.
    \end{enumerate}
\end{enumerate}

2) Assume $\mathbf Q$ satisfies clauses (a)-(e) of part (1).

Let $\bar q = \langle q_\varepsilon \colon \varepsilon < \delta\rangle$ be a
$\le^{\pr}_{\mathbf Q}$-increasing sequence of conditions, $\delta
< \theta$ a limit ordinal.  We say that $q$ is an
exact $\le^{\pr}_{\mathbf Q}$-upper bound of $\bar q$ \when \,
$\varepsilon < \delta = \ell g(\bar q) \Rightarrow q_\varepsilon
\le^{\pr}_{\mathbf Q} q$ and:

\begin{enumerate} 
\item[$(*)_{\bar q,q}$]   if $p \in \ap_{\mathbf Q}(q)$ \then \,
for some $\varepsilon < \delta$ and $p' \in 
\ap_{\mathbf Q}(q_\varepsilon)$, we have $\Vdash_{\mathbf Q}$ ``if $q,p'
\in \name{\mathbf G}_{\mathbf Q}$ then $p \in \name{\mathbf G}_{\mathbf Q}$".
\end{enumerate}
\end{definition}

\begin{remark}
\label{1c.17}  
Can we weaken clause (i) of $\circledast$
of \ref{1c.15}(1) to ``cardinality $\le \theta$"?

1) Here it mostly does not matter, \underline{but} in one point of the proof
of \ref{1c.25} it does: in proving $\circledast_4$ there, choosing
$\zeta(*)$ such that it will be possible to choose $\varepsilon(*)$.

2) There is a price for demanding a strict inequality.  The price is (in
\ref{2B.35}(1)) that, recalling $\kappa = \kappa_{\mathbf y}$, 
instead of using ap$_{\mathbf y}(q) = \{r \colon q
\le^{\ap}_\kappa r \in Q_{\mathbf y}\}$ we use
ap$_{\mathbf y}(q) = \{r \colon q \le^{\ap}_\kappa r \in Q_{\mathbf y}$ and
supp$_\kappa(q,r) \subseteq \text{ supp}_\theta
(p^{\mathbf y}_{\alpha_{\mathbf y}(q)},q)\}$.
\end{remark}

\begin{claim}
\label{1c.25}  
If $\mathbf Q$ is a $(\chi^+,\theta,\xi_*)$-forcing
notion, $\kappa < \theta = \text{\rm cf}(\theta)$ and 
$\chi = \chi^{< \theta}$ \then \, $\chi^+ \rightarrow 
(\xi_*,(\xi_*;\xi_*)_\kappa)^2$ holds in $\mathbf V^{\mathbf Q}$.
\end{claim}

\begin{remark}
We can replace $\chi^+$ by ``regular $\chi'$ such
that $\alpha < \chi' \Rightarrow |\alpha|^{< \theta} < \chi'$".
\end{remark}

\begin{PROOF}{\ref{1c.25}}
Let $\lambda_*$ be large enough (so in particular $\mathbf
Q,\theta,\dotsc, \in {\cH}(\lambda^+_*))$.  Choose a well ordering
$<^*_{\lambda^+_*}$ of the set $\cH(\lambda^+_*)$.
Recalling Definition \ref{1c.15} clearly $\theta > \aleph_0$, hence
\wilog \, $\kappa$ is infinite, so $1 + \kappa = \kappa$.

Toward contradiction assume $p^* \Vdash_{\mathbf Q} ``\name{\mathbf c}$ 
is a function from $[\chi^+]^2$ to $\kappa"$ is a counterexample.

We now choose $\bar M$ such that:

\begin{enumerate}
    \item[$\circledast_1$]   $(a) \quad \bar M = \langle M_\alpha \colon \alpha
    \le \theta\rangle$,
    
    \item[${{}}$]   $(b) \quad M_\alpha \prec ({\cH}(\lambda^+_*),\in),$
    
    \item[${{}}$]   $(c) \quad M_\alpha$ has cardinality $\chi$,
    
    \item[${{}}$]  $(d) \quad [M_\alpha]^{< \theta} \subseteq M_\alpha$
    if $\alpha$ is non-limit,
    
    \item[${{}}$]   $(e) \quad M_\alpha$ is $\prec$-increasing continuous,
    
    \item[${{}}$]   $(f) \quad \mathbf Q,p^*,\name{\mathbf c}$ belong to 
    $M_\alpha$ and $\chi +1 \subseteq M_\alpha$,
    
    \item[${{}}$]   $(g) \quad \bar M \restriction (\alpha +1) \in M_{\alpha +1}$. 
\end{enumerate}

Note that $\chi = \chi^{< \theta}$ implies $\theta < \chi^+$, so let

\begin{enumerate}
    \item[$\circledast_2$]   $\delta_* := \min(\chi^+ \backslash
    M_\theta)$.
\end{enumerate}

We shall now prove:

\begin{enumerate}
    \item[$\circledast_3$]   if $q \in \mathbf Q$ and $\varphi(x,y) \in 
    \bbL_{\theta,\theta}$ is a formula with parameters from $M_\theta$ such that
     $({\cH}(\lambda^+_*),\in,<^*_{\lambda^+_*}) \models
     \varphi[\delta_*,q]$ \then \, for some pair $(\delta,q') \in  
    M_\theta$ we have:
    
    \begin{enumerate}
        \item[$(a)$]   $\delta < \delta_*$,
        
        \item[$(b)$]  $({\cH}(\lambda^+_*),\in,<^*_{\lambda^+_*})
         \models \varphi[\delta,q']$,
        
        \item[$(c)$]  $q',q$ has a common $\le^{\pr}_{\mathbf Q}$-upper bound.
    \end{enumerate}
\end{enumerate}

Why $\circledast_3$ holds?  Let $\bar r = 
\langle r_\zeta \colon \zeta < \zeta^*\rangle$ list $\mathbf Q$,
each member appearing $\chi^+$ times,
now \wilog \, $\bar r \in M_0$ so necessarily we can find $\zeta_1 \in \zeta^*
\backslash M_\theta$ such that $q = 
r_{\zeta_1}$ and let $\zeta_2 = \min(M_\theta \cap
(\zeta^*+1) \backslash \zeta_1)$, of course $\zeta^* \in M_\theta$ and
$\zeta_2 \in M_\theta$ and $\zeta_1 < \zeta_2 \, \wedge$ cf$(\zeta_2) >\chi$.

Let
\[
Y = \{q' \in \mathbf Q \colon ({\cH}(\lambda^+_*),\in,<^*_{\lambda^+_*})
\models (\exists x)(\varphi(x,q') \wedge x \in \chi^+)\}.
\]

Recall that $\chi^{< \theta} = \chi$, so

\begin{enumerate}
    \item[$\odot_{3.1}$]   $Y \in M_\theta,Y \subseteq \mathbf Q$ and $q \in Y$.
\end{enumerate}

Now we ask:

\begin{enumerate}
    \item[$\odot_{3.2}$]   is there $Z \subseteq Y$ of cardinality
    $\le \chi$ such that for every $q'' \in Y$ for
    at least one $q' \in Z$ the pair $(q',q'')$ is
    $\le^{\pr}_{\mathbf Q}$-compatible? 
\end{enumerate}

Assume toward contradiction that the answer is negative, 
then in particular $|Y| > \chi$ and
we can choose $r_\varepsilon \in Y$ by induction on
$\varepsilon < \chi^+$ such that $\zeta < \varepsilon \Rightarrow$ the
pair $(r_\zeta,r_\varepsilon)$ is $\le^{\pr}_{\mathbf
Q}$-incompatible.   Why? In stage $\varepsilon$ try to use $Z :=
\{r_\zeta \colon \zeta < \varepsilon\}$, so $Z \subseteq Y$ has cardinality
$\le|\varepsilon| \le \chi$, so some $r_\varepsilon \in Y$ can serve as
$q''$ in $\odot_{3.2}$, by our assumption toward contradiction.  
Hence $\langle r_\varepsilon \colon \varepsilon <
\chi^+\rangle$ contradict clause (f) of Definition \ref{1c.15}(1).  So
the answer to $\odot_{3.2}$ is \underline{yes}, hence there is such $Z \in
M_\theta$, but $\chi +1 \subseteq M_\theta$ hence $Z \subseteq
M_\theta$.

So apply the property of $Z$, with $q$ standing for $q''$, so 
there is $q' \in Z \subseteq \mathbf Q \cap M_\theta$ such that
the pair $(q',q)$ is $\le^{\pr}_{\mathbf Q}$-compatible; but
$Z \subseteq Y$ hence by the definition of $Y$ there is 
$\delta \in \chi^+$ such that $({\cH}(\lambda^+_*),\in,<^*_{\lambda^+_*})
\models \varphi[\delta,q']$, and as
$q' \in Z \subseteq M_\theta$ \wilog \, $\delta \in M_\theta$,
hence $\delta \in \chi^+ \cap M_\theta$ so by the definition of $\delta_*$
we have $\delta < \delta_*$; so $\circledast_3$ holds indeed.

Next (but its proof will take awhile)

\begin{enumerate}
    \item[$\circledast_4$]   if $q^0 \in \mathbf Q$ is above $p^*$ 
    then for some triple $(q^1,p,\iota)$ we have:
    
    \begin{enumerate}
        \item[$(a)$]   $q^0 \le^{\pr}_{\mathbf Q} q^1,$
        
        \item[$(b)$]   $\iota < \kappa$,
        
        \item[$(c)$]   $p \in \ap_{\mathbf Q}(r)$ for some $r$
        satisfying $q^0 \le^{\pr}_{\mathbf Q} r \le^{\pr}_{\mathbf Q} q^1$,
        
        \item[$(d)$]   if $\iota = 0$ then $p \le q^1$,
        
        \item[$(e)$]   if $q$ satisfies $q^1 \le^{\pr}_{\mathbf Q} q$
        and $\varphi(x,y) \in \bbL_{\theta,\theta}$ is a formula with parameters from
        $M_\theta$ satisfied by the pair $(\delta_*,q)$ in the model 
        $({\cH}(\lambda^+_*),\in,<^*_{\lambda^+_*})$, \then \, we can find
        $q',q'',\delta$ such that the septuple $\mathbf q = (q,p,\iota,\varphi(x,y),
        q',q'',\delta)$ satisfies the following: 
        
        \item[${{}}$]    $\boxtimes_{\mathbf q} \quad \bullet_1 \quad 
        \delta < \delta_*$ (hence $\delta \in M_\theta$),
        
        \item[${{}}$]   $\quad \quad \bullet_2 \quad 
        ({\cH}(\lambda_*^+),\in,<^*_{\lambda^+_*}) \models \varphi[\delta,q']$, 
        
        \item[${{}}$]   $\quad \quad \bullet_3 \quad$ if $\iota = 0$ then:
        
        \item[${{}}$]  \hskip45pt $(\alpha) \quad 
        q \le^{\pr}_{\mathbf Q} q''$,
        
        \item[${{}}$]  \hskip45pt $(\beta) \quad q' \le^{\pr}_{\mathbf Q} q''$,
        
        \item[${{}}$]  \hskip45pt $(\gamma) \quad q'' 
        \Vdash ``\name{\mathbf c}\{\delta,\delta_*\} = 0"$.
        
        \item[${{}}$]  $\quad \quad \bullet_4 \quad$ if
        $\iota \in (0,\kappa)$, then $q \le^{\pr}_{\mathbf Q} q''$ and
        
        \hskip55pt $q'' \Vdash ``p \in 
        \name{\mathbf G}_{\mathbf Q} \Rightarrow \name{\mathbf c}
        \{\delta,\delta_*\} = \iota \wedge q' \in \name{\mathbf G}_{\mathbf Q}"$.
    \end{enumerate}
\end{enumerate}

Why?  Assume toward contradiction that $\circledast_4$ fails.
We let $\langle S_\varepsilon \colon \varepsilon \le \theta\rangle$ be
a $\subseteq$-increasing continuous sequence of subsets of $\theta$
with $S_\theta = \theta,|S_{\varepsilon +1} \backslash S_\varepsilon|
= \theta,|S_0| = \theta$ and min$(S_{\varepsilon +1} \backslash
S_\varepsilon) \ge \varepsilon$.
Now we try to choose $(q^*_\varepsilon,\mathbf x_\varepsilon,
\varphi_\varepsilon)$ by induction on
$\varepsilon < \theta$ (but $\varphi_\varepsilon$ is chosen in the
$(\varepsilon +1)$-th stage) such that:

\begin{enumerate}
    \item[$\odot_{4.1}$]   $(\alpha) \quad q^*_\varepsilon \in \mathbf Q$
    and $\langle q^*_\zeta \colon \zeta \le \varepsilon\rangle$
    is $\le^{\pr}_{\mathbf Q}$-increasing,
    
    \item[${{}}$]   $(\beta) \quad q^*_0 = q^0$,
    
    \item[${{}}$]  $(\gamma) \quad$ if $\varepsilon$ is a 
    limit ordinal $(< \theta)$ and $\langle q^*_\zeta \colon \zeta < \varepsilon\rangle$
    has an exact $\le^{\pr}_{\mathbf Q}$-upper
    
    \hskip25pt  bound (see part (2)
    of Definition \ref{1c.15}) \then \, $q^*_\varepsilon$ is an exact
    
    \hskip25pt  $\le^{\pr}_{\mathbf Q}$-upper bound of it,
    
    \item[${{}}$]   $(\delta) \quad \mathbf x_\varepsilon = \langle 
    (p^*_\xi,\iota_\xi) \colon \xi \in S_\varepsilon\rangle$ lists $\{(p,\iota) \colon 
    \iota < \kappa$ and 
    
    \hskip25pt $p \in \ap_{\mathbf Q}(q^*_\zeta)$ for some $\zeta$
    such that $\zeta = 0 \vee \zeta < \varepsilon\}$, here we use
    
    \hskip25pt  clause (i) of \ref{1c.15}(1) recalling $q^*_\zeta \in 
    \ap_{\mathbf Q}(q^*_\zeta)$, by clause (d)$(\beta)$ of 
    
    \hskip25pt \ref{1c.15}(1) so 
    $1 \le |\ap_{\mathbf Q}(q^*_\zeta)| < \theta$,
    
    \item[${{}}$]   $(\varepsilon) \quad$ for successor ordinal
    $\varepsilon = \zeta +1$, let $(q^*_{\zeta +1},\varphi_\zeta(x,y))$
    
    \hskip25pt  exemplify that the triple
    $(q^*_\zeta,p^*_\zeta,\iota_\zeta)$ does not satisfy 
    
    \hskip25pt  demand (e) on $(q^1,p,\iota)$ in $\circledast_4$, i.e.:
    
    \begin{enumerate}
        \item[$(*)$]   $q^*_\zeta \le^{\pr}_{\mathbf Q}
        q^*_{\zeta +1}$ and $\varphi_\zeta(x,y) \in 
        \bbL_{\theta,\theta}$ is a formula with 
        parameters from $M_\theta$ which the pair
        $(\delta_*,q^*_{\zeta +1})$ satisfies in 
        $({\cH}(\lambda^+_*),\in,<^*_{\lambda^+_*})$ but we cannot find
        $q',q'',\delta$ such that the septuple $\mathbf q_{\zeta +1} :=
        (q^*_{\zeta +1},p^*_\zeta,\iota_\zeta,\varphi_\varepsilon(x,y), 
        q',q'',\delta)$ satisfies $\boxtimes_{\mathbf q_{\zeta +1}}$.
    \end{enumerate}
\end{enumerate}

We show that the induction can be carried out.
Assume we are stuck at $\varepsilon$.  Now if $\varepsilon=0$ we
can satisfy clauses $(\alpha) + (\beta)$ and recalling
$1 \le |\ap_{\mathbf Q}(q^0)| < \theta$ we can choose $\mathbf x_0$ to
satisfy clause $(\delta)$ and since $(\gamma),(\varepsilon)$ are
vacuous we are done.

Suppose $\varepsilon >0$.  For limit
$\varepsilon$ we can choose $q^*_\varepsilon$ as required in clause
$(\alpha)$ by clause (e) of Definition \ref{1c.15}(1); also clause
$(\gamma)$ is relevant but causes no problem; and lastly, we can
choose $\mathbf x_\varepsilon$ and since clause $(\varepsilon)$ is vacuous
for limit ordinals, we are done again.  So $\varepsilon$ is a
successor, let $\varepsilon = \zeta +1$, so $q^*_\zeta$ was defined.
Now if we cannot choose $(q^*_{\zeta +1},\varphi_\zeta(x,y)) =
(q^*_\varepsilon,\varphi_\zeta(x,y))$ then the triple
$(q^*_\zeta,p^*_\zeta,\iota_\zeta)$ is as required from the 
triple $(q^1,p,\iota)$ in $\circledast_4$.  But this is impossible (by
our assumption toward contradiction), so we can find $(q^*_{\zeta
+1},\varphi_\zeta(x,y))$ as required; and again we can choose
$\mathbf x_\varepsilon$ as for $\varepsilon = 0$.

So it is enough to get a contradiction from the assumption that we can
carry out the induction.  But by clause (g) of Definition \ref{1c.15}(1) the
set $S := \{\zeta < \theta:\zeta$ is a limit ordinal and the sequence
$\langle q^*_\varepsilon \colon \varepsilon < \zeta\rangle$ has an exact
$\le^{\pr}_{\mathbf Q}$-upper bound$\}$ is
stationary.  As $S$ is stationary noting $\odot_{4.1}(\delta)$ 
and recalling clause (i) of Definition \ref{1c.15}(1) which gives
$|\ap_{\mathbf Q}(q^*_\varepsilon)| < \theta = \text{ cf}(\theta)$ for
$\varepsilon < \theta$, clearly for some limit ordinal 
$\zeta(*) \in S$ we have: if
$\iota < \kappa \, (< \theta)$ and $p \in 
\cup\{\ap_{\mathbf Q}(q^*_\varepsilon) \colon \varepsilon <
\zeta(*)\}$ then for unboundedly many $\varepsilon < \zeta(*)$ we have
$(p^*_\varepsilon,\iota_\varepsilon) = (p,\iota)$.

Let $\varphi(x,y) \in \bbL_{\theta,\theta}$ express all 
the properties that the pair
$(\delta_*,q^*_{\zeta(*)})$ satisfies and are used below, i.e., $(\exists
y_0,\dotsc,y_{\zeta(*)})[x \in \chi^+ \wedge y = y_{\zeta(*)} \wedge
\bigwedge\limits_{\varepsilon < \zeta \le \zeta(*)} y_\varepsilon
\le^{\pr}_{\mathbf Q} y_\zeta \wedge \bigwedge\limits_{\varepsilon
< \zeta(*)} \varphi_\varepsilon(x,y_{\varepsilon +1}) \wedge
(y_{\zeta(*)}$ is an exact $\le^{\pr}_{\mathbf Q}$-upper bound of
$\langle y_i \colon i < \zeta(*)\rangle )]$.

So

\begin{enumerate}
\item[$(*)$]  $({\cH}(\lambda^+_*),\in,<^*_{\lambda^+_*})
\models \varphi[\delta_*,q^*_{\zeta(*)}]$.
\end{enumerate}
By $\circledast_3$ we can find a pair $(\delta,q')$ such that:

\begin{enumerate}
    \item[$\odot_{4.2}$]   $(a) \quad \delta < \delta_*$ hence $\delta 
    \in M_\theta$ and $q' \in M_\theta$,
    
    \item[${{}}$]   $(b) \quad ({\cH}(\lambda^+_*),\in,<^*_{\lambda^+_*}) \models
    \varphi[\delta,q']$,
    
    \item[${{}}$]  $(c) \quad q',q^*_{\zeta(*)}$ are
    $\le^{\pr}_{\mathbf Q}$-compatible.
\end{enumerate}

Let $q''$ be such that

\begin{enumerate}
\item[${{}}$]   $(d) \quad q' \le^{\pr}_{\mathbf Q} q''$ and
$q^*_{\zeta(*)} \le^{\pr}_{\mathbf Q} q''$.
\end{enumerate}

Let $\langle q'_\zeta \colon \zeta \le \zeta(*)\rangle$ exemplify
$\varphi[\delta,q']$ and \wilog \, $\{q'_\zeta \colon \zeta \le
\zeta(*)\} \subseteq M_\theta$, in particular, $\varepsilon \le
\zeta(*) \Rightarrow q'_\varepsilon \le^{\pr}_{\mathbf Q} 
q'_{\zeta(*)} = q' \le^{\pr}_{\mathbf
Q} q''$ and, of course, $\varepsilon \le \zeta(*) \Rightarrow
q^*_\varepsilon \le^{\pr}_{\mathbf
Q} q^*_{\zeta(*)} \le^{\pr}_{\mathbf Q} q''$.

\underline{Case 1}:  $q'' \Vdash_{\mathbf Q} ``\name{\mathbf
c}\{\delta,\delta_*\} = 0"$.

There is $\varepsilon < \zeta(*)$ such that $\iota_\varepsilon = 0$.
We get contradiction to the choice of the $(q^*_{\varepsilon
+1},\varphi_\varepsilon)$.

Why?  Let us check that the septuple
$\mathbf q = (q^*_{\varepsilon +1},q^*_{\varepsilon
+1},0,\varphi_\varepsilon(x,y),q'_{\varepsilon +1},q'',\delta)$ is
such that $\boxtimes_{\mathbf q}$ holds.

\underline{For $\bullet_1$}:  Recall $\odot_{4.2}(a)$

\underline{For $\bullet_2$}:  By $\odot_{4.1}(\varepsilon)(*)$ we have
$(\cH(\lambda^+_*),\in,<^*_{\lambda^+_*}) \models
\varphi_\varepsilon(\delta_*,q^*_{\varepsilon +1})$ by the choice of
$\varphi(x,y)$ and of $\langle q'_\zeta \colon \zeta \le \zeta(*)\rangle$ we
have $(\cH(\lambda^+_*),\in,<^*_{\lambda^+_*}) \models
\varphi_\varepsilon[\delta,q'_{\varepsilon +1}]$ as required.

\underline{For $\bullet_3(\alpha)$}:  it means $q^*_{\varepsilon +1}
\le^{\pr}_{\mathbf Q} q''$ which holds as $q^*_{\varepsilon +1}
\le^{\pr}_{\mathbf Q} q^*_{\zeta(*)}$ by $\odot_{4.1}(\alpha)$ and
$q^*_{\zeta(*)} \le^{\pr}_{\mathbf Q} q''$ by $\odot_{4.2}(d)$.

\underline{For $\bullet_3(\beta)$}:   it means $q'_{\varepsilon +1}
\le^{\pr}_{\mathbf Q} q''$ which has been proved just before
``Case 1".

\underline{For $\bullet_3(\gamma)$}:  it means $q'' \Vdash
``\name{\mathbf c}\{\delta,\delta_*\}=0"$ which holds by the case
assumption

\underline{For $\bullet_4$}:  it is vaccuous.

So indeed $\boxtimes_{\mathbf q}$ holds contradicting the choice of
$(q^*_{\varepsilon +1},\varphi_\varepsilon)$, see $\odot_{4.1}(\varepsilon)$.

\underline{Case 2}:  Not Case 1.

Choose $(q^+,\iota)$ such that $q^+ \in \mathbf Q,q^*_{\zeta(*)}
\le_{\mathbf Q} q'' \le_{\mathbf Q} q^+$
and $q^+ \Vdash_{\mathbf Q} ``\name{\mathbf c}\{\delta,\delta_*\} =
\iota"$ where $\iota \in (0,\kappa)$, we use ``not Case 1".  By clause
(j) of $\circledast$ of Definition \ref{1c.15} applied with
$(q^*_{\zeta(*)},q^+)$ here standing for $(q_*,r)$ there, we
can find a pair $(s,p)$ such that

\begin{enumerate}
    \item[$\odot_{4.3}$]   $(a) \quad p \in \ap_{\mathbf Q}(q^*_{\zeta(*)})$,
    
    \item[${{}}$]   $(b) \quad q^*_{\zeta(*)} \le^{\pr}_{\mathbf Q} s$,
    
    \item[${{}}$]   $(c) \quad s \Vdash_{\mathbf Q} ``p \in 
    \name{\mathbf G}_{\mathbb Q} \Rightarrow q^+ \in \name{\mathbf G}_{\mathbb Q}"$.
\end{enumerate}

As $q^*_{\zeta(*)}$ is an exact $\le^{\pr}_{\mathbf Q}$-upper
bound of $\langle q^*_\varepsilon \colon \varepsilon <
\zeta(*)\rangle$ because $\zeta(*) \in S$ 
and $p \in \ap_{\mathbf Q}(q^*_{\zeta(*)})$, see
part (2) of Definition \ref{1c.15}, there is a pair
$(p',\varepsilon(*))$ such that:

\begin{enumerate}
\item[$\odot_{4.4}$]   $(a) \quad \varepsilon(*) < \zeta(*)$,

\item[${{}}$]   $(b) \quad p' \in \ap_{\mathbf Q}(q^*_{\varepsilon(*)})$,

\item[${{}}$]  $(c) \quad \Vdash_{\mathbf Q} ``\text{if }
q^*_{\zeta(*)},p' \in \name{\mathbf G}_{\mathbf Q}$ then 
$p \in \name{\mathbf G}_{\mathbf Q}"$.
\end{enumerate}

So by the choice of $\zeta(*)$ for some $\zeta < \zeta(*)$ which is $>
 \varepsilon(*)$ we have $(p^*_\zeta,\iota_\zeta) = (p',\iota)$.  
Let $\mathbf q = (q^*_{\zeta +1},p^*_\zeta,\iota_\zeta,
\varphi_\zeta(x,y),q'_\zeta,s,\delta)$.  This
septuple satisfies $\boxtimes_{\mathbf q}$ because:

\underline{For $\bullet_1$}:  Recall $\odot_{4.2}(a)$.

\underline{For $\bullet_2$}:  as in case 1.

\underline{For $\bullet_3$}:  it is vaccuous.

\underline{For $\bullet_4$}:   it means first $q^*_{\zeta +1}
\le^{\pr}_{\mathbf Q} s$ which holds as $q^*_{\zeta +1}
\le^{\pr}_{\mathbf Q} q^*_{\zeta(*)}$ by $\odot_{4.1}(\alpha)$ and
$q^*_{\zeta(*)} \le^{\pr}_{\mathbf Q} s$ by $\odot_{4.3}(b)$.
Second, $s \Vdash ``p^*_\zeta \in \name{\mathbf G}_{\mathbf Q} \Rightarrow
\name{\mathbf c}\{\delta,\delta_*\} = \iota"$ which holds as $p^*_\zeta = p'$
and assuming $\mathbf G \subseteq \bbQ$ is generic over $\mathbf V$ if
$s,p' \in \mathbf G$ then by $\odot_{4.3}(b)$ also $q^*_{\zeta(*)} \in
\mathbf G$ hence by $\odot_{4.4}(c)$ also $p \in \mathbf G$ hence by
$\odot_{4.3}(c)$ also $q^+ \in \mathbf G$ hence by the choice of $q^+$
in the beginning of the case we have $\mathbf V[\mathbf G]$ satisfies 
$\name{\mathbf c}[\mathbf G]\{\delta,\delta_*\} = \iota$.

Third, $s \Vdash ``p^*_\zeta \in \name{\mathbf G}_{\mathbf Q} \Rightarrow
q'_\zeta \in \name{\mathbf G}_{\mathbf Q}"$ which holds as $p^*_\zeta = p'$
and assuming $\mathbf G \subseteq \bbQ$ is generic over $\mathbf V$, if
$s,p' \in \mathbf G$ then as above $q^+ \in \mathbf G$ hence by 
the choice of $q^+$ in the beginning of the case 
also $q'' \in \mathbf G$ hence by
$\odot_{4.2}(d)$ also $q' \in \mathbf G$ hence by the choice of $\varphi$
and of $\langle q'_\zeta \colon \zeta \le \zeta(*)\rangle$ we have $q'_\zeta
\in \mathbf G$ as required.

Hence we get a contradiction to
the choice of $(q^*_{\zeta +1},\varphi_3)$.  
So we are done proving $\circledast_4$.

Let the triple $(q_*,p_*,\iota_*)$ satisfy the demands on
$(q^1,p,\iota)$ in $\circledast_4$ for $q^0 = p^*$ and let $r_*$ be as
guaranteed by clause (c) of $\circledast_4$ so

\begin{enumerate}
\item[$\odot$]  $p_* \le^{\pr}_{\mathbf Q} r_*
\le^{\pr}_{\mathbf Q} q_*$ and $p \in \ap_{\mathbf Q}(r_*)$.
\end{enumerate}

Now we choose
$q_\zeta,q'_\zeta,q''_\zeta,q'''_\zeta,r_\zeta,p_\zeta,
\alpha_\zeta,\beta_\zeta$ by induction on $\zeta < \theta$ such that:

\begin{enumerate}
\item[$\circledast_5$]  $(a) \quad q_\zeta \in \mathbf Q$,

\item[${{}}$]  $(b) \quad \langle q_\xi:\xi \le \zeta\rangle$ is
$\le^{\pr}_{\mathbf Q}$-increasing,
 
\item[${{}}$]  $(c) \quad q_0 = q_*$, 

\item[${{}}$]   $(d) \quad \alpha_\zeta < \beta_\zeta < \delta_*$
and $\varepsilon < \zeta \Rightarrow \beta_\varepsilon < \alpha_\zeta$,

\item[${{}}$]   $(e) \quad (q'_\zeta,q''_\zeta,\alpha_\zeta)$ 
is as $(q',q'',\delta)$ is guaranteed to be in clause (e) of
$\circledast_4$ 

\hskip25pt with $q_\zeta$ here standing for $q$ there (and of course
$p_*,\iota_*$ here stands

\hskip25pt for $p,\iota$ there) and a suitable $\varphi$, hence,

\begin{enumerate}
\item[${{}}$]   $(\alpha) \quad \alpha_\xi,\beta_\xi <
\alpha_\zeta < \delta_*$ for $\xi < \zeta$,

\item[${{}}$]   $(\beta) \quad q_\zeta \le^{\pr}_{\mathbf Q} q''_\zeta$,

\item[${{}}$]   $(\gamma) \quad$ the pair $(\alpha_\zeta, q'_\zeta) 
\in M_\theta$ is similar enough to $(\delta_*,q_\zeta)$,

\item[${{}}$]   $(\delta) \quad$ if $\iota_* > 0$ then
$q''_\zeta \Vdash$ ``if $p_* \in
\name G_{\mathbf Q}$ then $\name{\mathbf c}\{\alpha_\zeta,\delta_*\} = \iota_*$
and $q'_\zeta \in \name{\mathbf G}_{\mathbf Q}"$,

\item[${{}}$]  $(\varepsilon) \quad$ if $\iota_* =0$ then $q'_\zeta
\le^{\pr}_{\mathbf Q} q''$ and $q''_\zeta \models ``\name{\mathbf
c}\{\alpha_\zeta,\delta_*\} = \iota_*$ and $q''_\zeta \models
``\name{\mathbf c}\{\alpha_\varepsilon,\alpha_\zeta\} = \iota_*"$ for
$\varepsilon < \zeta$.

\end{enumerate}
\item[${{}}$]    $(f) \quad$ the quadruple
$(\beta_\zeta,r_\zeta,p_\zeta,q'''_\zeta) \in M_\theta$ is similar enough to
the 

\hskip25pt quadruple $(\delta_*,r,p_*,q''_\zeta)$, i.e.:

\begin{enumerate}
    \item[${{}}$]   $(\alpha) \quad
    \beta_\zeta \in (\alpha_\zeta,\delta_*)$,
    
    \item[${{}}$]  $(\beta) \quad$ the pair $(q'''_\zeta,q''_\zeta)$
    is $\le^{\pr}_{\mathbf Q}$-compatible,
    
    \item[${{}}$]   $(\gamma) \quad p_\zeta \in \ap_{\mathbf Q}
    (r_\zeta)$ and $r_\zeta \le^{\pr}_{\mathbf Q} q'''_\zeta$, 
    
    \item[${{}}$]    $(\delta) \quad q'''_\zeta \Vdash_{\mathbf Q}$ ``if 
    $p_\zeta \in \name{\mathbf G}_{\mathbf Q}$ then
    $\name{\mathbf c}\{\alpha_\varepsilon,\beta_\zeta\} = 
    \iota_*$ for $\varepsilon \le \zeta$".
\end{enumerate}

\item[${{}}$]   $(g) \quad q''_\zeta \le^{\pr}_{\mathbf Q}
q_{\zeta +1}$ and $q'''_\zeta \le^{\pr}_{\mathbf Q} q_{\zeta +1}$.
\end{enumerate}

[Why can we carry out the induction?  Note that
$q'_\zeta,\dotsc,\beta_\zeta$ are chosen in the $(\zeta +1)$-th step.

For $\zeta=0$ just let $q_0 = q_*$ so the only relevant clauses
(a),(c) are satisfied.

For $\zeta$ limit only clause (b) is relevant and we can choose
$q_\zeta$ by clause (e) of Definition \ref{1c.15}.

We are left with $\zeta$ successor, let $\zeta = \xi +1$.  

We first choose $(q'_\xi,q''_\xi,\alpha_\xi)$ as required in clause (e) of
$\circledast_5$ using appropriate $\varphi$ and 
$\circledast_4(e)$ for our $(q_*,p_*,\iota_*)$.  
Clearly in $\circledast_5$ clause (e) holds as well as the second
statement in clause (d).  In particular, $(e)(\delta)$ comes from
$\circledast_4(e)$, and $(e)(\varepsilon)$ comes from $\varphi$,
i.e. as $\varepsilon < \zeta \Rightarrow q_\varepsilon
\le^{\pr}_{\mathbf Q} q_\zeta$.

Second, we choose $(\beta_\xi,r_\xi,
p_\xi,q'''_\xi)$ as required in clause (f) of
$\circledast_5$.  [Why?  We can find $(\beta_\xi,r_\xi,p_\xi,q'''_\xi)
\in M_\theta$ similar enough to $(\delta_*,r,p_*,q''_\xi)$, using
$(*)_3$ with $(\delta_*,q''_\zeta)$ here standing for $(\delta_*,q)$
there and $q'''_\zeta$ here standing for $q'$ in the conclusion of
$\circledast_3$ (and $r_\xi,p_\xi$ are gotten by existential
quantifiers in choosing $\varphi$ which holds as $r_*,p_*$ witness). 

First, note that $\alpha_\zeta < \delta_*$ holds as 
$\alpha_\zeta \in M_\theta$ hence $\alpha$ and $\beta_\xi < \delta_*$ but
$\beta_\xi \in M_\theta$ so $\beta_\xi < \delta_*$ so clause
$(f)(\alpha)$ holds.  Second, $q'''_\zeta,q''_\zeta$ are
$\le^{\pr}_{\mathbf Q}$-compatible by $\circledast_3(c)$ hence clause
$(f)(\beta)$ holdsa.

Third, the parallel of $(f)(\gamma)$ holds for $(p_*,r_*)$ by the
choice of $r_*$ and as $q_* = q_0 \le^{\pr}_{\mathbf Q} q_\zeta
\le^{\pr}_{\mathbf Q} q''_\zeta$.

Fourth, the parallel of $(f)(\delta)$ holds for $(q''_\zeta,p_*)$ by
$(e)(\delta)$.  

Third, as $q''_\xi,q'''_\xi$ are $\le^{\pr}_{\mathbf
Q}$-compatible there is $q_\zeta = q_{\xi +1}$ as required in clause (g).

So we can satisfy $\circledast_5$.

Now we apply clause (h) of Definition \ref{1c.15}(1) to the sequence
$\langle(q_\varepsilon,p_\varepsilon) \colon \varepsilon < \theta\rangle$
hence there is $\zeta < \theta$ as there, so as $p_\varepsilon \in
\ap_{\mathbf Q}(q_\varepsilon)$ the conditions
$p_\varepsilon,q_\varepsilon$ are compatible in $\bbQ$ hence they have
a common upper bound $r \in \mathbf Q$ hence by the choice of $\langle
(p_\varepsilon,q_\varepsilon) \colon \varepsilon < \theta\rangle$ above,
 $r \Vdash_{\mathbf Q} ``\xi_* \le \text{ otp}\{\varepsilon <
\zeta \colon q_\varepsilon,p_\varepsilon \in \name{\mathbf G}_{\mathbf Q}\}"$.

So $r \Vdash_{\mathbf Q}$ ``the sequence
$\langle(\alpha_\varepsilon,\beta_\varepsilon) \colon \varepsilon < \zeta$
and $q_\varepsilon,p_\varepsilon \in 
\name{\mathbf G}_{\mathbf Q}\rangle$ is as required" noting that:

\begin{enumerate}
    \item[$\bullet$]  if $\iota_* \ge 0$, then $q_{\zeta +1} \models
    ``\name{\mathbf c}\{\alpha_\varepsilon,\beta_\zeta\}" = \iota_*$ for
    $\varepsilon \le \zeta$,
    
    \item[$\bullet$]  if $\iota_* = 0$, then $q_{\zeta +1} \models
    ``\name{\mathbf c}\{\alpha_\varepsilon,\alpha_\zeta\} = \iota_*$ for
    $\varepsilon \le \zeta$.
\end{enumerate}

So we are done.  
\end{PROOF}

%%%%%%%%%%%%%%%%%%%%%%%%%%%%%%%%%%%%%%%%%%%%%%%%%%%%%%%%
% Section 2
\newpage

\section{Many strong polarized partition relations}\label{2}  

We can below say more on strongly inaccessible $\theta \in \Theta$.

\begin{hypothesis}\label{2B.48}  
    Let $\mathbf p = (\lambda,\mu,\Theta,\bar\partial)$ satisfy:
    
    \begin{enumerate}
        \item[$(a)$]   $\lambda =\lambda^{<\lambda} < \mu = \mu^{< \mu},$
        
        \item[$(b)$]    $\Theta \subseteq [\lambda,\mu]$ is a set of regular
        cardinals with $\lambda,\mu \in \Theta,$
        
        \item[$(c)$]    $\bar\partial = \langle \partial_\theta \colon \theta \in
        \Theta\rangle$ is an increasing sequence of cardinals such that: 
        
        \begin{enumerate}
            \item[$(\alpha)$]   $\partial_\theta = \text{ cf}(\partial_\theta),$
            
            \item[$(\beta)$]    $\partial_\theta = (\partial_\theta)^{< \partial_\theta},$
            
            \item[$(\gamma)$]   $\partial_\theta \le \theta$ and if $\theta < \kappa$ are from $\Theta$ then $\partial_\theta < \partial_\kappa,$
            
            \item[$(\delta)$]  $\partial_\theta \ge \kappa$ if $\kappa \in (\Theta \cap \theta),$
            
            \item[$(\varepsilon)$]   if $\theta = \lambda$ then
            $\partial_\theta = \lambda$.
        \end{enumerate}

        \item[$(d)$] notation: let $\theta, \Upsilon, \kappa$ vary on $\Theta.$ 
    \end{enumerate}
\end{hypothesis}

The reader may concentrate on (see \ref{3c.87}):

\begin{example}\label{2B.49}  
    Assume
    
    \begin{enumerate}
        \item[$(a)$]   $\mathbf V$ satisfies G.C.H. from $\lambda$ to $\mu$,
        i.e., $\partial \in [\lambda,\mu) \Rightarrow 2^\partial = \partial^+,$
        
        \item[$(b)$]   $\lambda = \lambda^{< \lambda} < \mu = \mu^{< \mu},$
        
        \item[$(c)$]    $\Theta := \{\theta^+:\lambda \le \theta < \mu\}    \cup\{\lambda,\mu\}$ and,
        
        \item[$(d)$]   $\partial_\theta = \theta$ for every $\theta \in        \Theta$, so in \ref{2B.85}(5)        below in this example we have $\partial^\theta = \min\{\theta^+,\mu\}$.
    \end{enumerate}
\end{example}

For the rest of this section $\mathbf p$, i.e. $\lambda,\mu,\Theta,\bar\partial$ are fixed.

\begin{example}\label{2B.52}
    As above \blueq{by} $$\Theta \coloneqq \{ \kappa \colon \kappa = \lambda \text{ or } \kappa = \mu \text{ or } \kappa = \lambda^{\omega(1 + \alpha) + 1} < \mu \text{ for some } n < \} $$
\end{example}

\begin{definition}\label{2B.85}  
    1) For $\kappa \in \Theta$, 
    let $E_\kappa$ be the equivalence relation on $\mu$ defined by
    
    \begin{enumerate}
        \item[$(*)$]    $i E_\kappa j$ iff $i + \kappa = j + \kappa$.
    \end{enumerate}
    
    2) For any cardinal $\kappa \in [\lambda,\mu]$ 
    define $E_{<\kappa}$ as {\rm Eq}$_\lambda \cup 
    \bigcup \{E_\theta \colon \theta \in \Theta \cap
    \kappa\}$.     For such $\kappa$, if $\kappa \notin \Theta$, let $E_\kappa = E_{<\kappa}$.

    3)  For $i < \mu$ and $\kappa \in \Theta$ 
    let $[i]_\kappa = i/E_\kappa =$ the $E_\kappa$-equivalence 
    class of $i$, and for 
    $A \subseteq \mu$, let $A/E_\kappa = \{i/E_\kappa \colon i \in A\}$.  For 
    $i < \mu,A \subseteq \mu$ we say that $i/E_\kappa$ is 
    \underline{represented in $A$} \underline{iff} 
    $A \cap (i/E_\kappa) \ne \emptyset$.  If $A \subseteq B \subseteq \mu$,
    we say that $i/E_\kappa$ \underline{grows from} $A$ to $B$
    \underline{iff} $\emptyset \ne A 
    \cap (i/E_\kappa) \ne B \cap (i/E_\kappa)$.  If
    we write functions $p,q$ instead of $A,B$ we mean $\dom(p)$, $\dom(q)$
    respectively.

    4) Note that for all $i,j < \mu$ we have $i E_\mu j$.  
    Thus, the following definition makes sense: \underline{if} $i,j$ are 
    $< \mu$ we let $\kappa(i,j)$ be the minimal $\kappa
    \in \Theta$ such that $i E_\kappa j$.

    5) Suppose $\kappa \in \Theta$, let
    \[
    \partial^\kappa = \min\{\partial_\theta \colon \kappa < \theta 
    \in \Theta\} \text{ if } \kappa < \mu \text{ and } \partial^\kappa =
    \mu \text{ if } \kappa = \mu.
    \]
    
    (Notice that $\kappa$ is just an index in $\partial^\kappa$, and this
    is not cardinal exponentiation.)
\end{definition}

Thus, in particular,

\begin{observation}\label{2B.12}  
    1) For $i,j < \mu$ we have: $\kappa(i,j)$
    is well defined and for $i,j < \mu,\theta \in [\lambda,\mu)$ 
    we have $i E_\theta j \Leftrightarrow \theta \ge \kappa(i,j)$ as: 
    
    \begin{enumerate}
        \item[$(*)$]   if $\theta < \kappa$ are both from $\Theta$, \then \,
        $E_\theta$ refines $E_\kappa$ and, in fact, each $E_\kappa$-equivalence 
        class is
        the union of $\kappa$ many $E_\theta$-equivalence classes.
    \end{enumerate}
    
    2a) If $\kappa < \theta$ are from $\Theta$ 
    then $\partial^\kappa \le \partial_\theta$; used in \ref{2B.24}(1).
    
    2b) $\partial_\theta < \partial^\theta$
     except possibly for $\theta = \mu$ (still
    $\partial_\mu \le \mu = \partial^\mu$); recall
     \ref{2B.48}(c)$(\gamma)$.
    
    2c) {\rm sup}$(\Theta \cap \kappa) \le \partial_\kappa$
    for $\kappa \in \Theta$; recall \ref{2B.48}$(c)(\delta)$
    
    2d) $\partial^\theta = (\partial^\theta)^{< \partial^\theta}$ 
    for $\theta \in \Theta;$ recall \ref{2B.48}(c)$(\beta),$ 
    
    2e)  If $\kappa \in \Theta$ \then \,
    each $E_{< \kappa}$-equivalence class has
    cardinality $\le \partial_\kappa$ (by (2c)); used in the proof 
    of \ref{2B.24}(3)).
    
    3a)  $\partial_\lambda = \lambda$.
    
    3b)  If $\theta < \kappa$ are successive elements of    $\Theta$ \then \, $\partial^\theta = \partial_\kappa$.
    
    3c)  If $\kappa \in \Theta$ and 
    $\bigcup (\Theta \cap \kappa)$ is a singular cardinal, \then \,
    $\partial_\kappa \ge (\bigcup (\Theta \cap \kappa))^+$.
\end{observation}

\begin{definition}\label{2B.16}  
    1) The forcing notion $\bbQ \coloneqq \bbQ_{\mathbf p} = (\bbQ, \leq_{\bbQ}) = (Q_{\mathbf p},\le_{\bbQ_{\mathbf p}})$, but we may omit $\mathbf p$ when clear 
    from the context, is defined by:
    
    \begin{enumerate}
        \item[$(A)$]   $q \in Q$ \underline{iff}
        \begin{enumerate}
        
        \item[$(a)$]     $q$ is a (partial) function from $\mu$ to $\{0,1\}$,
        
        \item[$(b)$]  if $i < \mu$ and $\kappa \in \Theta$, then the
        cardinality of $(i/E_\kappa) \cap \dom(q)$ is $< \partial_\kappa$ 
        (note: taking $\kappa = \mu$, the cardinality of $\dom(q)$ is 
        $< \partial_\mu \le \mu$),
    \end{enumerate}
    
    \item[$(B)$] $p \leq q$ or $p \le_{\bbQ} q$ \underline{iff}
    \begin{enumerate}
    \item[$(a)$]   $p \subseteq q$, i.e. $\dom(p) \subseteq 
    \dom(q)$ and $\alpha \in \dom(p) \Rightarrow  p(\alpha) =
    q(\alpha)$ 
    
    \item[$(b)$]   for every $\theta \in \Theta$ the set $\{A \in \mu
    /E_\theta \colon A$ grows from $p$ to $q\}$ has cardinality $< \partial_\theta$.
    \end{enumerate}
    \end{enumerate}
    
    2) For $\kappa \in \Theta \backslash \{\mu\}$ and $p,q \in Q$, let:
    
    \begin{enumerate}
        \item[$(A)$]   $p \le^{\pr}_{\mathbf p,\kappa} q$ \text{or} 
        $p \le^{\pr}_\kappa q$ \underline{iff}:
        
        \begin{enumerate}
            \item[$(a)$]    $p \le q$ and,
            
            \item[$(b)$]    no $E_\kappa$-equivalence class grows from $p$ to $q$.
        \end{enumerate}
        
        \item[$(B)$]   $p \le^{\ap}_{\mathbf p,\kappa} q$ or
        $p \le^{\ap}_\kappa q$ \underline{iff}:
        
        \begin{enumerate}
        \item[$(a)$]    $p \le q$,
        
        \item[$(b)$]   $\dom(q)/E_\kappa = \dom(p)/E_\kappa$.
        \end{enumerate}
    \end{enumerate}
    
    3) For $\kappa = \mu$ and $p,q \in Q$, let:
    
    \begin{enumerate}
        \item[$(A)$]   $p \le^{\pr}_\mu q$ \text{iff} $p = q.$
        
        \item[$(B)$]   $p \le^{\ap}_\mu q$ iff $p \le q$.
    \end{enumerate}
    
    4) Let $\mathbf Q_\kappa = \mathbf Q_{\mathbf p,\kappa} = (Q,\le_{\bbQ},
    \le^{\pr}_\kappa,\ap_\kappa)$ where
    {\rm ap}$_\kappa = \text{\rm ap}_{\mathbf p,\kappa}$ is the function
    with domain $Q$ such that {\rm ap}$_\kappa(q) = \{q':q
    \le^{\ap}_\kappa q'\}$; so $\mathbf Q_\kappa$ as a forcing notion is
    $\bbQ$.
    
    5) Let $\le^{\text{us}}_{\mathbf p,\kappa} = \le^{\text{us}}_\kappa =
    \le_{\mathbf p}$ be $\le_{\bbQ_{\mathbf p}}$ for $\kappa \in \Theta$.
\end{definition}

\begin{remark}
    Clearly $\mathbf Q_\kappa$ is related to \S1, and if
    $\kappa$ is the last member of $\Theta \cap \mu$ we can use it (enough
    if $\Theta = \{\lambda,\mu\}$, \underline{but} not in general,
    so we shall use a variant).
\end{remark}

\begin{claim}\label{2B.20}  
Concerning Definition \ref{2B.16}

\begin{enumerate}
\item[$(a)$]   $(\alpha) \quad$ if $\kappa \in \Theta$, \then \,
$\le,\le^{\text{\rm pr}}_\kappa,
\le^{\text{\rm ap}}_\kappa$ are partial orderings of $Q$,

\item[${{}}$]   $(\beta) \quad p \le^{\text{\rm pr}}_\kappa q
\Rightarrow p \le q$ and $p \le^{\text{\rm ap}}_\kappa q \Rightarrow p \le q$,

\item[${{}}$]   $(\gamma) \quad$ if $\kappa = \mu$ \then \,
$\le^{\text{\rm ap}}_\kappa = \le$,

\item[${{}}$]   $(\delta) \quad$ if $\kappa = \mu$ \then \,
$\le^{\text{\rm pr}}_\kappa$ is the equality.

\item[$(b)$]   $(\alpha) \quad$ if $p_1,p_2 \in Q$ and they 
are compatible as functions, \then \, $p_1 \cup p_2 \in Q$;

\item[${{}}$]   $(\beta) \quad$ moreover, letting 
$q = p_1 \cup p_2$, if clause (b) of \ref{2B.16}(1)(B) holds

\hskip25pt  between $p_k$ and $q$, for $k = 1,2$, \then \, $q$
is the lub, in $\bbQ$, of 

\hskip25pt  $p_1$ and $p_2$.

\item[$(c)$]   if $p \le q$ and $\kappa \in \Theta$, \then \, there are 
$r,s \in Q$ such that:
\begin{enumerate}
\item[$(\alpha)$]  $p \le^{\text{\rm pr}}_\kappa r 
\le^{\text{\rm ap}}_\kappa q$,

\item[$(\beta)$]  $p \le^{\text{\rm ap}}_\kappa s 
\le^{\text{\rm pr}}_\kappa q,$ 

\item[$(\gamma)$]   $q = r \cup s$,

\item[$(\delta)$]   $q$ is the $\le$-lub of $r,s$.
\end{enumerate}

\item[$(d)$]   if $q \in Q$ \then:

\begin{enumerate}
\item[$(\alpha)$]  $\emptyset \le q$ (and $\emptyset$, the empty
function, $\in Q_{\mathbf p}$),

\item[$(\beta)$]  $(\forall r)(q \le r \equiv q \le^{\text{\rm ap}}_\mu r)$,

\item[$(\gamma)$]    $\kappa \in \Theta \backslash \{\mu\}
\Rightarrow \emptyset \le^{\text{\rm pr}}_\kappa q$,

\item[$(\delta)$]   $\emptyset \ne q \Rightarrow
\emptyset \nleq^{\text{\rm ap}}_\kappa q$ for 
any $\kappa \in \Theta \backslash \{\mu\}.$
\end{enumerate}

\item[$(e)$]   if $\kappa_1 \le \kappa_2$ are both from 
$\Theta$, \then:

$\le^{\text{\rm pr}}_{\kappa_2} \subseteq \le^{\text{\rm pr}}_{\kappa_1}$ and
$\le^{\text{\rm ap}}_{\kappa_1} \subseteq \le^{\text{\rm ap}}_{\kappa_2},$

\item[$(f)$]   if $\kappa \in \Theta$ and $p \le^{\text{\rm ap}}_\kappa q$
and $p \le^{\text{\rm pr}}_\kappa r$, \then \,: 
\begin{enumerate}
\item[$(\alpha)$]  $q \cup r$ is a well defined function $\in Q,$

\item[$(\beta)$]  $p \le (q \cup r),$

\item[$(\gamma)$]  $q \le^{\text{\rm pr}}_\kappa (q \cup r),$

\item[$(\delta)$]  $r \le^{\text{\rm ap}}_\kappa (q \cup r),$

\item[$(\varepsilon)$]   $q \cup r$ is a $\le$-lub of $q,r$ in
$\bbQ_{\mathbf p}.$
\end{enumerate}

\item[$(g)$]   if $\kappa \in \Theta, p \le^{\text{\rm pr}}_\kappa q_i 
\, (i = 1,2)$ and $q_1,\ q_2$ are compatible in $\mathbf Q$ (even just as
functions), \then \, $p \le^{\text{\rm pr}}_\kappa \, (q_1 \cup q_2),$

\item[$(h)$]   if $p \le^{\text{\rm ap}}_\kappa q_k$ for $k = 1,2$, and
$q_1,q_2$ are compatible in $\mathbf Q$ (even just as functions),
\then \, $q_k \le^{\text{\rm ap}}_\kappa (q_1 \cup q_2)$ for $k=1,2,$

\item[$(i)$]   $(\alpha) \quad$ if 
$\{p_\varepsilon \colon \varepsilon < \zeta\}$ has an
$\le$-upper bound \then \,
$\cup\{p_\varepsilon \colon \varepsilon < \zeta\}$ is an 

\hskip25pt upper bound,

\item[${{}}$]  $(\beta) \quad$ similarly for
$\le^{\text{\rm pr}}_\kappa,\le^{\text{\rm ap}}_\kappa,$

\item[${{}}$]  $(\gamma) \quad$ assume $p_\varepsilon \in Q$ for every
$\varepsilon < \zeta$, and $p_\varepsilon,p_\xi$ has a common 

\hskip25pt  $\le^{\text{\rm x}}_\kappa$-upper bound for any 
$\varepsilon,\xi < \zeta$; \then \, the union of

\hskip25pt  $\{p_\varepsilon \colon \varepsilon < \zeta\}$ is a 
$\le^{\text{\rm x}}_\kappa$-lub

\hskip35pt when $x = \text{\rm us,ap}$ and $\zeta < \lambda$

\item[${{}}$]  $(\delta) \quad$ if $\{p_\varepsilon \colon \varepsilon <
\zeta\} \subseteq Q$ has a common $\le^{\text{\rm pr}}_\kappa$-upper bound
and $\zeta < \partial^\kappa$, \then \, 

\hskip35pt $\{p_\varepsilon \colon \varepsilon <
\zeta\}$ has a $\le^{\text{\rm pr}}_\kappa$-lub - the union

\item[$(j)$]   if $p \le^{\text{\rm ap}}_\kappa q$ \then \, {\rm Dom}$(q)
\backslash \text{\rm Dom}(p)$ has cardinality $< \partial_\kappa$,

\item[$(k)$]   if $p_1 \le^{\text{\rm ap}}_\kappa p_3$ and $p_1 \le p_2
\le p_3$ \then \, $p_1 \le^{\text{\rm ap}}_\kappa p_2$ and 
$p_2 \le^{\text{\rm ap}}_\kappa p_3$,

\item[$(l)$]   if $p_1 \le^{\text{\rm pr}}_\kappa p_2,p_\ell 
\le^{\text{\rm ap}}_\kappa q_\ell$ for $\ell=1,2$ and $q_1 \cup q_2$
is a function, \then \, $q := q_1 \cup q_2$ is a $\le$-lub of
$q_1,q_2$ and $q_2 \le^{\text{\rm ap}}_\kappa q,q_1 \le q$,

\item[$(m)$]   assume $p_1,p_2$ are compatible in $\bbQ$ \then \,
there is a pair $(q,t)$ such that:

\begin{enumerate}
    \item[$\bullet_1$]  $p_1 \le^{\text{\rm pr}}_\kappa q$,
    
    \item[$\bullet_2$]  $p_2 \le^{\text{\rm ap}}_\kappa t$,
    
    \item[$\bullet_3$]   $q \Vdash ``t \in \name{\mathbf G} \Rightarrow
    p_1 \in \name{\mathbf G}"$,
    
    \item[$\bullet_4$]   $q,t$ are compatible
    and we say $(q,t)$ is a witness for $(p_1,p_2).$
\end{enumerate}

\item[$(n)$]   if $\langle p^\ell_\alpha \colon \alpha < \delta\rangle$ is
$\le^{\text{\rm pr}}_\kappa$-increasing for $\ell=1,2,\delta$ a limit
ordinal of cofinality $< \partial_\kappa$ and $\alpha < \delta
\Rightarrow p^1_\alpha \le^{\text{\rm ap}}_\kappa p^2_\alpha$ then
$\bigcup\limits_{\alpha < \delta} p^1_\alpha \le^{\text{\rm ap}}_\kappa
(\bigcup\limits_{\alpha < \delta} p^2_\alpha)$.
\end{enumerate}
\end{claim}

\begin{PROOF}{\ref{2B.20}}

Straightforward. E.g.

\underline{Clause (i)}:

So assume $x \in \{\text{us, pr, ap}\}$ and $\kappa \in \Theta$ and
$\{p_\varepsilon \colon \varepsilon < \zeta\} \subseteq Q$ and $q \in Q$
is an $\le^{\text{\rm x}}_\kappa$-upper bound of $\{p_\varepsilon \colon \varepsilon
<\zeta\}$.  Let $p := \cup\{p_\varepsilon \colon \varepsilon < \zeta\}$
\then \,  we shall prove that $p \in Q$ and $p$ is a
$\le^{\text{\rm x}}_\kappa$-upper bound 
of $\{p_\varepsilon \colon \varepsilon < \zeta\}$;
this clearly suffices for proving sub-clauses $(\alpha),(\beta)$ of 
clause (i), and the $\le^x_\kappa$-lub part, i.e. sub-clauses
$(\gamma),(\delta)$ are left to the reader; 
for $(\gamma),(\delta)$, see \ref{2B.24}(1B),(1A).

Now

\begin{enumerate}
\item[$(*)_1$]   $p$ is a well defined function with domain
$\subseteq \mu$ and $p \subseteq q$.
\end{enumerate}

[Why?  As $\varepsilon < \zeta \Rightarrow p_\varepsilon \subseteq q$,
i.e. as functions (by \ref{2B.16}(1)(B)(a)) clearly $p \subseteq q$,
as functions, so $p$ is a well defined function with domain $\subseteq
\dom(q)$ but $\dom(q) \subseteq \mu$ by \ref{2B.16}(A)(a).]

\begin{enumerate}
\item[$(*)_2$]   if $i < \mu$ and $\theta \in \Theta$ \then  \, the
cardinality of $(i/E_\theta) \cap \dom(p)$ is $< \partial_\theta$.
\end{enumerate}

[Why?  Recall $p \subseteq q \in Q$, see above so as $q \in Q$ by
 \ref{2B.16}(1)(a) we have $|(i/E_\theta) \cap 
\dom(p)| \le |(i/E_\theta) \cap \dom(q)| < \partial_\theta$.]

\begin{enumerate}
\item[$(*)_3$]    $p \in Q$.
\end{enumerate}

[Why?  By $(*)_1 + (*)_2$ recalling \ref{2B.16}(1)(A).]

\begin{enumerate}
    \item[$(*)_4$]    $p_\varepsilon \subseteq p$ for $\varepsilon < \zeta$.
\end{enumerate}

[Why?  By the choice of $p$.]

\begin{enumerate}
\item[$(*)_5$]    if $\varepsilon < \zeta$ and $\theta \in \Theta$
then $\{A \in \mu/E_\theta \colon A$ grows from $p_\varepsilon$ to $p\}$ has
cardinality $< \partial_\theta$.
\end{enumerate}

[Why?  Because, recalling $p \subseteq q$, this set is included in
$\{A \in \mu/E_\theta \colon A$ grows from $p_\varepsilon$ to $q\}$ which has
cardinality $< \partial_\theta$ because $p_\varepsilon \le q$ which
holds as $p_\varepsilon \le^x_\kappa q$.]

\begin{enumerate}
\item[$(*)_6$]    $p_\varepsilon \le p$ for $\varepsilon < \zeta$.
\end{enumerate}

[Why?  By $(*)_4 + (*)_5$ recalling \ref{2B.16}(1)(B).]

\begin{enumerate}
\item[$(*)_7$]    if $x = \text{ us}$ \then \, $p$ is a $\le$-upper
bound of $\{p_\varepsilon \colon \varepsilon < \zeta\}$.
\end{enumerate}

[Why?  By $(*)_3 + (*)_6$.]

\begin{enumerate}
\item[$(*)_8$]    if $x = \text{ pr}$ and $\varepsilon < \zeta$
\then \, $p_\varepsilon \le^{\pr}_\kappa p$.
\end{enumerate}

[Why?  If $\kappa = \mu$ then $\le^{\pr}_\kappa$ is equality and
$p_\varepsilon \le^{\pr}_\kappa q$ hence $p_\varepsilon = q$ but
$p_\varepsilon \subseteq p \subseteq q$ hence $p_\varepsilon = p$ 
so this is trivial, hence assume $\kappa < \mu$.
We have to check \ref{2B.16}(2)(A), now clause (a) there
holds by $(*)_6$ and clause (b) there holds as no
$E_\kappa$-equivalence class grows from $p_\varepsilon$ to $q$ (as
$p_\varepsilon \le^{\pr}_\kappa q$) and $p \subseteq q$.]

\begin{enumerate}
\item[$(*)_9$]   if $x = \text{ pr}$ then $p$ is a
$\le^{\text{\rm x}}_\kappa$-upper bound of 
$\{p_\varepsilon \colon \varepsilon < \zeta\}$.
\end{enumerate}

[Why?  By $(*)_8$.]

\begin{enumerate}
\item[$(*)_{10}$]   if $x = \ap$ and $\varepsilon < \zeta$
then $p_\varepsilon \le^{\ap}_\kappa p$.
\end{enumerate}

[Why?  If $\kappa = \mu$ then $\le^{\ap}_\kappa =
\le^{\text{us}}_\kappa$ and we are done by $(*)_7$.  Assume $\kappa < \mu$.  
We have to check \ref{2B.16}(2)(B).  First, clause (a) there
holds by $(*)_6$.  Second, clause (b) there holds because if $A \in
\dom(p)/E_\kappa$ then $A \cap \dom(p) \ne \emptyset$ by
the definition, hence $A \cap \dom(q) \ne \emptyset$ as $p
\subseteq q$ by $(*)_1$, but this implies $A \cap \text{
Dom}(p_\varepsilon) \ne \emptyset$ because $p_\varepsilon
\le^{\ap}_\kappa q$, as required.]

\begin{enumerate}
    \item[$(*)_{11}$]    if $x = \ap$ then $p$ is a
    $\le^{\text{\rm x}}_\kappa$-upper bound 
    of $\{p_\varepsilon \colon \varepsilon < \zeta\}$.
\end{enumerate}

[Why?  By $(*)_{10}$.]

The $\le^{\text{\rm x}}_\kappa$-lub parts 
are easy too, for a limit ordinal $\delta$ see \ref{2B.24}(1A).

\underline{Clause (j)}:

Let ${\cU} = \{A \colon A \in \mu/E_\kappa$ and $A$ grows from $p$ to $q\}$.  Recalling Definition \ref{2B.16}(1)(B)(b) clearly, as $p \le q$, we have $|{\cU}| < \partial_\kappa$.  But as $p \le^{\ap}_\kappa
q$ necessarily $\dom(q) \backslash \dom(p)$ is included in $\cup\{A \colon A \in {\cU}\}$.  Also as $q \in Q$ by Definition  \ref{2B.16}(1)(A)(b) we have $A \in {\cU} \Rightarrow |A \cap  \dom(q)| < \partial_\kappa$.

So $\dom(q) \backslash \dom(p)$ is included in 
$\cup\{A \cap \dom(q) \colon A \in {\cU}\}$, a union of  $< \partial_\kappa$ sets each of cardinality $< \partial_\kappa$. But $\partial_\kappa$ is regular by \ref{2B.48}(C)$(\beta)$, so we are
done.] 

\underline{Clause (m)}: As $p_1,p_2$ are compatible in $\bbQ$, there is $r \in \bbQ$ such that $p_1 \le r,p_2 \le r$.  Choose $t = \cup \{r \rest (i/E_\kappa) \colon i/E_\kappa$ grow from $p_2$ to $r\} \cup p_2$, so $t \in Q$ and $p_2 \le^{\text{\rm ap}}_\kappa t \le^{\text{\rm pr}}_\kappa r$. Choose $q = \cup\{r \rest (i/E_\kappa) \colon i/E_\kappa$ does not grow from $p_1$ to $r\} \cup r$, so $q \in Q$ and $p_1  \le^{\text{\rm pr}}_\kappa q \le^{\text{\rm apr}}_\kappa r$.  

Now check.
\end{PROOF}

\begin{claim}\label{2B.24}  
    Let $\kappa \in \Theta$.
    
    1) $(Q,\le^{\text{\rm pr}}_\kappa)$ 
    is $(< \partial^\kappa)$-complete and in
    fact if $\bar p = \langle p_\alpha \colon \alpha <
    \delta\rangle$ is $<^{\text{\rm pr}}_\kappa$-increasing, $\delta$ a limit
    ordinal $< \partial^\kappa$ \then \, $p_\delta := \cup\{p_\alpha \colon \alpha <
    \delta\}$ is a $\le^{\pr}_\kappa$-lub and $a \le$-lub of $\bar p$; 
    we use $\kappa < \theta \in \Theta \Rightarrow 
    \partial^\kappa \le \partial_\theta$, see \ref{2B.12}(2a).  
    
    1A) If $\gamma(*) < \partial^\kappa$ and $p_\alpha \in Q$ for $\alpha
    < \gamma(*)$ and $p_\alpha,p_\beta$ has a common
    $\le^{\text{\rm pr}}_\kappa$-lub for any $\alpha,\beta < \gamma(*)$ \then
    \, $p_* = \cup\{p_\alpha \colon \alpha < \gamma(*)\}$ is a
    $\le^{\text{\rm pr}}_\kappa$-lub of $\{p_\alpha \colon \alpha < \gamma(*)\}$.
    
    1B) If $\gamma(*) < \lambda$ then (1A) holds for $\le^{\ap}_\kappa$.
    
    2) If $k \in \Theta$ and $p \in \bbQ$ \then \, $\bbQ_{p, k} := \bbQ_{\mathbf p,p, k} =  (\{q \colon p \le^{\text{\rm ap}}_\kappa q\},
    <^{\text{\rm ap}}_\kappa)$  satisfies\footnote{compare with \cite[1.8]{Sh:608}}
    the $(\partial_\kappa)^+$-c.c.
    
    3) Moreover if $\langle p_\alpha \colon \alpha < \partial^+_\kappa\rangle$ is
    $\le^{\text{\rm pr}}_\kappa$-increasing continuous and $p_\alpha \le^{\text{\rm ap}}_\kappa q_\alpha$ for $\alpha < \partial_\kappa^+$, 
    \then \, for some $\alpha < \beta$ the conditions $q_\alpha,q_\beta$  are compatible in $\mathbb Q,$ moreover there is $r$ such that $q_\alpha
    \le r$ and $q_\beta \le^{\text{\rm ap}}_\kappa r$ and $p_\alpha = p_\beta \Rightarrow q_\alpha \le^{\text{\rm ap}}_\kappa r \wedge q_\beta
    \le^{\text{\rm ap}}_\kappa r$.
    
    4) Assume $p \in \bbQ_{\mathbf p},\chi = |A| < \partial^\kappa,\kappa \in
    \Theta$ and $p \Vdash ``\name f$ is a function from $A \in \mathbf V$ to $\mathbf V$".  \Then \, we can find $q$ such that:
    
    \begin{enumerate}
    \item[$(\alpha)$]   $p \le^{\text{\rm pr}}_\kappa q,$
    
    \item[$(\beta)$]   if $a \in A$ then ${\cI}_{q,\name f,a} :=  \{r \colon q \le^{\text{\rm ap}}_\kappa r$ and $r$ forces a value to $\name f(a)\}$ is predense over $q$ in $\bbQ_{\mathbf q},$
    
    \item[$(\gamma)$]   moreover some subset 
    ${\cI}'_{q,\name f,a}$ of ${\cI}_{q,\name f,a}$ of cardinality $\le \partial_\kappa$ is predense over $q$ in $\bbQ_{\mathbf q},$ (really
    follows).
    \end{enumerate}
\end{claim}

\begin{PROOF}{\ref{2B.24}}
    1) By (1A).
    
    1A) Let $q_{\alpha,\beta}$ be a common
    $\le^{\pr}_\kappa$-upper bound of $p_\alpha,p_\beta$ for
    $\alpha,\beta < \gamma(*)$.
    Why is $p_* \in Q$?  Let us check Definition \ref{2B.16}(1)(A).
    
    Clearly $p_*$ is a partial function from $\mu$ to $\{0,1\}$ so
    clause (a) there holds.  For checking clause (b) there, assume $\theta \in
    \Theta$ and $A \in \mu/E_\theta$.  
    
    First, assume $\theta \le \kappa$
    and $A \cap \dom(p_*) \ne \emptyset$ then for some $\alpha
    < \gamma(*)$ we have $A \cap \dom(p_\alpha) \ne \emptyset$, 
    hence $A \cap \dom(p_*) = \cup\{A \cap \text{
    Dom}(p_\beta) \colon \beta < \gamma(*)\} \subseteq \cup\{A \cap 
    \dom(q_{\alpha,\beta}) \colon \beta < \gamma(*)\}$, but $p_\alpha
    \le^{\text{\rm pr}}_\kappa q_{\alpha,\beta}$ and 
    $A \cap \dom(p_\alpha) \ne \emptyset$ hence
    $A \cap \dom(q_{\alpha,\beta}) = A \cap \text{
    Dom}(p_\alpha)$.  Together 
    $A \cap \dom(p_*)$ is equal
    to $A \cap \dom(p_\alpha)$ which, because $p_\alpha \in Q$,  
    has cardinality $< \partial_\theta$ as required in clause (b) of Definition 
    \ref{2B.16}(1)(A).  
    
    Second, of course, if $A \cap \dom(p_*) =
    \emptyset$ this holds, too.
    
    Third, assume $\theta > \kappa$, then $\alpha < \gamma(*) \Rightarrow
    p_\alpha \in Q \Rightarrow |A \cap \dom(p_\alpha)| <
    \partial_\theta$, hence $|A \cap \dom(p_*)| = |A \cap
    \bigcup\limits_{\alpha < \gamma(*)} \dom(p_\alpha)| \le 
    \sum\limits_{\alpha < \gamma(*)} |A \cap \dom(p_\alpha)|$ which 
    is $< \partial_\theta$ as $\gamma(*) < \partial^\kappa \le \partial_\theta = 
    \text{ cf}(\partial_\theta)$, so again the desired conclusion of clause (b) of
    Definition \ref{2B.16}(1)(A) holds.  Together indeed $p_* \in Q$. 
    
    Why $\alpha < \gamma(*) \Rightarrow p_\alpha \le p_*$?  We have to
    check \ref{2B.16}(1)(B), obviously clause (a) there holds.  
    Clause (b) there is proved as above.
    
    Why $\alpha < \gamma(*) \Rightarrow p_\alpha \le^{\pr}_\kappa
    p_*$?  We have to check Definition \ref{2B.16}(2)(A), now
    clause (a) there was just proved and clause (b) there holds as in the
    proof of ``$p_* \in Q$".
    
    Next we show that $p_*$ is a $\le^{\pr}_\kappa$-lub of
    $\bar p$, so assume $q \in Q$ and $\alpha < \delta \Rightarrow
    p_\alpha \le^{\pr}_\kappa q$.  To show $p_*
    \le^{\pr}_\kappa q$ we have to check clauses (B)(a),(b) of
    \ref{2B.16}(1) and (A)(b) of \ref{2B.16}(2).  As $p_* =
    \cup\{p_\alpha \colon \alpha < \gamma(*)\}$, clearly $p_*
    \subseteq q$ as a function so \ref{2B.16}(1)(B)(a) above holds.  
    Also if $A \in \mu/E_\kappa$ and $A$ is represented in $p_*$ then it is
    represented in $p_\alpha$ for some $\alpha < \gamma(*)$, but $p_\alpha
    \le^{\pr}_\kappa q$ so $q \restriction A = p_\alpha
    \restriction A$ but $(p_\alpha \rest A) \subseteq (p_* \rest A) \subseteq (q
    \rest A)$ hence $q \restriction A = p_* \restriction A$ as
    required in \ref{2B.16}(2)(A)(b).
    
    Lastly, when $\theta \in \Theta$, \ref{2B.16}(1)(B)(b) holds: if
    $\theta \le \kappa$ because more was just proved and if $\theta > \kappa$ it
    is proved as in the proof of $p_* \in Q$.

    2) This is a special case of (3) when $\langle p_\alpha \colon \alpha
    < \partial^+_\kappa\rangle$ is constant (recalling \ref{2B.20}(h)).

    3) So in particular $p_i \le^{\ap}_\kappa q_i$ for $i <
    \partial_\kappa^+$.  Hence by clause (j) of Claim \ref{2B.20} the set
    $u_i := \dom(q_i) \backslash 
    \dom(p_i)$ has cardinality $< \partial_\kappa$.  Hence
    by the $\Delta$-system lemma (recalling that $(\partial_\kappa)^{<
     \partial_\kappa} = \partial_\kappa$ by \ref{2B.48}(c)$(\beta)$) 
    for some unbounded ${\cU} \subseteq
    \partial_\kappa^+$ the sequence $\langle u_i \colon i \in {\cU}\rangle$ 
    is a $\Delta$-system, with heart $u_*$.  Moreover, since $2^{|u_*|}
    \le \partial^{< \partial_\kappa}_\kappa = \partial_\kappa <
    \partial^+_\kappa$, we can assume that $q_i \rest u_* = q_*$ for every
    $i \in \cU$.
    
    As each $E_{< \kappa}$-class has cardinality $\le \partial_\kappa$
    (see \ref{2B.12}(2)(c),(e)), without loss of generality for
    every $i \ne j$ from ${\cU}$, if $\alpha \in u_i \backslash u_*$
    then $\alpha/E_{< \kappa}$ is disjoint to $u_j$.  Now by \ref{2B.20}(h) 
    for every $i,j \in {\cU}$, the function $q=q_i \cup q_j$ is a
    $\le^{\ap}_\kappa$-lub of $q_i,q_j$ for part (2), i.e. when $p_i
    = p_j$.  Also it is easy to check that for $i<j,q$ is a
    $\le$-lub of $q_i,q_j$ which is $\le^{\ap}_\kappa$-above $q_j$
    for part (3).

    4) If $\kappa = \mu$ then $\le^{\ap}_\kappa = \le$ by clause
    \ref{2B.20}(a)$(\gamma)$, recall $\bbQ_p =  (\{q \in Q \colon p \le q\},
    \le_{\bbQ_{\mathbf p}})$ so $q=p$ can serve, as $\bbQ_p$ satisfies 
    the $\partial^+_\kappa$-c.c. by part (2); so we shall assume $\kappa < \mu$. 
    Recall that $\partial_\kappa < \partial^\kappa$ by \ref{2B.12}(2)(b).
     As $|A| < \partial^\kappa = \text{ cf}(\partial^\kappa)$, 
    by part (1) of the claim and clause (f) of Claim \ref{2B.20} it 
    is enough to consider the case $A = \{a\}$.  Now we try to choose 
    $p_i,r_i,b_i$ by induction on $i < \partial_\kappa^+$, but $r_i,b_i$
    are chosen in stage $i+1$ together with $p_{i+1}$, such that:
    
    \begin{enumerate}
    \item[$\circledast$]   $(a) \quad p_0 = p$,
    
    \item[${{}}$]   $(b) \quad \langle p_j \colon j \le i\rangle$ is
    $\le^{\pr}_\kappa$-increasing,
    
    \item[${{}}$]   $(c) \quad p_{i+1} \le^{\ap}_\kappa r_i$,
    
    \item[${{}}$]   $(d) \quad p_{i+1} \Vdash$ ``if $r_i \in
    \name{\mathbf G}_{\mathbf Q}$ then $\name f(a) = b_i$",
    
    \item[${{}}$]   $(e) \quad p_{i+1} \Vdash$ ``if $r_i \in
    \name {\mathbf G}_{\mathbf Q}$ then for no $j<i$ do we
    have $r_j \in \name {\mathbf G}_{\mathbf Q}$",
    
    \item[${{}}$]  $(f) \quad$ if $i$ is a limit, then $p_i$ is the union
      so a $\le^{\pr}_\kappa$-lub of $\langle p_j \colon j<i\rangle$.
    \end{enumerate}
    
    For $i=0$ just use clause (a) of $\circledast$.
    
    For $i$ limit use clause (f) of $\circledast$ recalling part (1) of
    the claim and the fact that $\partial^+_\kappa \le \partial^k$.
    
    For $i=j+1$, try to choose $q_i$ such that:
    \[
    p_j \le q_i
    \]
    
    and
    \[
    q_i \Vdash ``r_{i_1} \notin \name {\mathbf G}_{\mathbf Q} \text{ for } i_1 < j".
    \]
    
    If we cannot, we have succeeded, i.e. $p_i$ is as required from $q$ with
    ${\cI}_{p_i,\name f,a} = \{p_i \cup r_j \colon j<i\}$.   
    If we can, let $(b_j,r_j)$ be such
    that $q_i \le r_j$ and $r_j$ forces $\name f(a) =
    b_j$; clearly possible. By clause (c) of Claim \ref{2B.20} applied
    to the pair $(p_j,r_j)$ we choose\footnote{we can use $r'_j$ such that
    $p_j \le^{\ap}_\kappa r'_j \le^{\pr}_\kappa r_j$ such that
    $r_j$ is the $\le$-lub of $r'_j,p_{i+1}$, may be helpful but not
    needed now.} $p_i$ such that $p_j
    \le^{\pr}_\kappa p_i \le^{\ap}_\kappa r_j$ and clearly we
    have carried out the induction.  But if we carry the induction then we get
    a contradiction by part (3).  So we have to be stuck for some $i <
    \partial_\kappa^+$, and as said above we then get the desired
    conclusion.  
\end{PROOF}

\begin{conclusion}\label{2B.29}   
    Forcing with $\bbQ_{\mathbf p}$:
    
    \begin{enumerate}
        \item[$(a)$]   does not collapse cardinals except possibly cardinals
        from the set $\Omega_{\mathbf p}$ = $\{\theta \colon \lambda < \theta \le \mu$ 
        and for no $\kappa \in \Theta$ do we have $\partial_\kappa < \theta \le 
        \partial^\kappa\}$, so $\mu \notin \Omega_{\mathbf p}$,
        
        \item[$(b)$]   does not change cofinalities $\notin \Omega_{\mathbf p}$,
        moreover if it changes the cofinality of $\theta \in \text{ Reg}$ to
        $\chi < \theta$ then there is $\theta_1 \in \Omega_{\mathbf p}$ such
        that $\chi \leq \theta < \theta_1,$ moreover, $[\chi, \theta_{1}] \cap \rm{Reg} \subseteq \Omega_{\mathbf{p}}$,
        
        \item[$(c)$]   does not add new sequences of length $< \lambda$,
        
        \item[$(d)$]  does not change $2^\theta$ for $\theta \notin [\lambda,\mu)$,
        
        \item[$(e)$]   makes $2^\lambda = \mu$,
        
        \item[$(f)$]   also the set $\Omega'_{\mathbf p} := \cup\{(\kappa_1,
        2^{\text{sup}(\Theta \cap \kappa)}]$: for some $\kappa \in \Theta,\Theta
        \cap \kappa$ has no last member, so sup$(\Theta \cap \kappa)$ is
        strong limit and $\kappa_1 = \min(\text{Reg} \backslash
        \sup(\Theta \cap \kappa))\}$, is O.K. in clauses (a),(b),
        
        \item[$(g)$]   $\bbQ_{\mathbf p}$ has cardinality $\mu$ and satisfies the
        $\partial^+_\mu$-c.c., recalling $\partial_\mu \le \mu$.
    \end{enumerate}
\end{conclusion}

\begin{PROOF}{\ref{2B.29}}
    First, $\bbQ_{\mathbf p}$ is $(< \lambda)$-complete hence it adds no new sequences to ${}^{\lambda >}\mathbf V$, i.e. clause
    (c) holds so cardinals $\le \lambda$ are preserved as well as cofinalities $\le \lambda$ as well as $2^\theta$ for $\theta <
    \lambda$.
    
    Second, $|\bbQ_{\mathbf p}| = \mu$ as $p \in \bbQ_{\mathbf p} \Rightarrow p$
    is a function from $\dom(p) \subseteq \mu$ to $\{0,1\}$, see \ref{2B.16}(1)(A)(a) and $|\dom(p)| < \partial_\mu = \mu$ by \ref{2B.16}(1)(A)(b) and $\mu^{< \mu} = \mu$ by \ref{2.1}(a).
    
    Third, by \ref{2B.24}(2) the forcing notion $\bbQ_p$ satisfies the $\partial^+_\mu$-c.c. but $\bbQ = \bbQ_p$ when $p = \emptyset$ so $\bbQ$ satisfies the $\partial^+_\mu$-c.c. and of course $\partial_\mu
    \le \mu$.  This gives clauses (g) and (d) (recalling (c)).  
    
    Fourth, for clause (e), for any $\alpha < \mu$ let $\name \eta_\alpha \in {}^\lambda 2$ be defined by $p \Vdash ``\eta_\alpha(i) = \ell"$ \underline{iff} $i < \lambda \wedge \char 94 \alpha + i \in \dom(p)
    \wedge \char 94 \ell = p(\alpha +i)$.  By density indeed $\Vdash_{\bbQ} ``\name \eta_\alpha \in {}^\lambda 2"$ and $\Vdash_{\bbQ} ``\name \eta_\alpha \ne \name \eta_\beta"$ for $\alpha
    \ne \beta < \mu$, so clearly clause (e) holds.  
     
    Fifth, use \ref{2B.24}(2),(4) to prove clauses (a) and (b), toward contradiction assume $\theta$ is regular in $\mathbf V$ and $\theta_1$ is not in
    $\Omega_{\mathbf p}$ but $p \Vdash_{\bbQ} ``\chi = \text{ cf}(\theta) < \theta_1 \le \theta"$.  If $\theta \le \lambda$ or just $\chi < \lambda$
     use clause (c), if $\theta > \mu$ use clause (g) so necessarily $\lambda \le \chi < \theta_1 \le \theta \le \mu$.  By the choice of $\Omega_{\mathbf p}$
    there is $\kappa \in \Theta$ such that $\partial_\kappa < \theta_1 \le \partial^\kappa$ and $\chi + \partial_\kappa < \theta_1 \le \theta$; 
    now \wilog \, $p \Vdash ``\name f \colon \chi
    \rightarrow \theta$ has range unbounded in $\theta"$. Apply \ref{2B.24}(4) with $(p,\chi,\name f,\kappa)$ here standing for $(p,A,\name f,\kappa)$ there and get $q,\langle \cI_{q,\name f,\alpha} \colon \alpha < \chi\rangle$ as there.  By \ref{2B.24}(3) we have
    $|\cI_{q,\name f,\alpha}| \le \partial_\kappa$ and $\cup\{\cI_{q,\name f,\alpha} \colon \alpha < \chi\}$ has cardinality $\le \chi + \partial_\kappa <
    \theta_1$.  In any case, in $\mathbf V$ the set
    $\{\beta$ :  for some $\alpha < \chi$ and $q \nVdash ``\name f(\alpha) \ne \beta"\}$ has cardinality $< \theta_1 \le \theta$, contradiction.
    So clauses (a),(b) holds.
    
    We are left with clause (f), it is not really needed, still nice to have.  Now if $\theta \in \text{ Reg} \cap(\lambda,\mu]$ is in $\Omega'_{\mathbf p}$ and $\kappa$ witness it  then necessarily $\Theta \cap \kappa$, which is not empty has no last element so if $\theta_1 < \theta_2$ are from $\Theta \cap \theta$ then $\theta_1 \le \partial_{\theta_2} =
    (\partial_{\theta_2})^{<\partial_{\theta_2}} \le \theta_2$ hence sup$(\Theta \cap \theta)$ is strong limit.  
    
    If $\theta = \kappa$ use clause (b).  If $\theta \ge 2^{< \kappa}$ we repeat the proofs above for $\le^{\pr}_{< \kappa}$ where $\le^{\pr}_{< \kappa} = \cap\{\le^{\pr}_\theta \colon \theta \in
    \Theta \cap \theta\},\le^{\ap}_{< \kappa} = \{(p,q) \colon p \le q$ and $\alpha \in$ $\dom(p) \backslash \dom(p) \Rightarrow (\exists \theta \in \Theta \cap \theta)((\alpha/E_\theta \cap \dom(p))
    \ne \emptyset\}$.
\end{PROOF}

\begin{definition}\label{2B.33}
1) If $p \le q$ and $\kappa \in \Theta$ let supp$_\kappa(p,q) := 
\cup\{i/E_\kappa \colon i \in \dom(q) \backslash \dom(p)\}$ so
of cardinality $< \partial_\kappa$.

2) We say $\mathbf y = \langle \kappa,\bar p,\bar u\rangle = \langle
\kappa_{\mathbf y},\bar p_{\mathbf y},\bar u_{\mathbf y}\rangle$ is a
reasonable $\mathbf p$-parameter \when:

\begin{enumerate}
    \item[$\circledast_1$]  $(a) \quad \kappa \in \Theta$ but $\kappa < \mu,$ 
    
    \item[${{}}$] $ (b) \quad \theta = \theta_{\mathbf{y}} = \min(\Theta \setminus \kappa_{\mathbf{y}}^{+}),$ notice that $\theta$ is well defined, as $\kappa_{\mathbf{y}} < \mu$ and $\mu \in \Theta$,      
    
    \item[${{}}$]  $(c) \quad \bar p = \langle p_\alpha \colon \alpha < \gamma\rangle$ is
    a non-empty $\le^{\pr}_\theta$-increasing continuous sequence, 
    
    \hskip25pt so we 
    write $\gamma = \gamma_{\mathbf y},\bar p = \bar p^{\mathbf y}$ 
    and $p_\alpha = p^{\mathbf y}_\alpha,$
    
    \item[${{}}$]  $(c) \quad \bar u = \langle u_\alpha \colon \alpha < \gamma\rangle$ is
    $\subseteq$-increasing continuous, so $u_\alpha = u^{\mathbf
    y}_\alpha,\bar u = \bar u_{\mathbf y},$
    
    \item[${{}}$]  $(d) \quad u_\alpha \subseteq \cup\{i/E_\kappa \colon i \in \text{
    Dom}(p_\alpha)\}$ for $\alpha < \gamma,$
    
    \item[${{}}$]  $(e) \quad |u_\alpha| \le \partial^\kappa$ 
    for $\alpha < \gamma$.
\end{enumerate}

3) For $\mathbf y$ as above we define $\mathbf Q_{\mathbf y}$ as
$(Q_{\mathbf y},\le_{\mathbf y},\le^{\pr}_{\mathbf y},
\ap_{\mathbf y})$ (so $\bbQ_{\mathbf y} = (Q_{\mathbf y},\le_{\mathbf y})$ is
$\mathbf Q_{\mathbf y}$ as a forcing notion), where:

\begin{enumerate}
\item[$\circledast_2$] 

\item[${{}}$]   $(a) \quad Q_{\mathbf y} := \{q$: for some
$\alpha < \gamma_{\mathbf y}$ we have $p_\alpha \le^{\ap}_\theta q$
 and supp$_\theta(p_\alpha,q) \subseteq u_\alpha\},$

\item[${{}}$]   $(b) \quad \le_{\mathbf y}
= \le_{\mathbf p} \restriction Q_{\mathbf y}$,

\item[${{}}$]   $(c) \quad$ for $q \in Q_{\mathbf y}$, let
$\alpha_{\mathbf y}(q) = \min\{\alpha < \gamma_{\mathbf y} \colon p_\alpha
\le^{\ap}_\theta q$ and

\hskip25pt  supp$_\theta(p_\alpha,q) \subseteq u_\alpha\}$,

\item[${{}}$]   $(d) \quad$ the two-place relation 
$\le^{\pr}_{\mathbf y}$ is defined by $p \le^{\pr}_{\mathbf y} q$ iff:
\begin{enumerate}

\item[${{}}$]  $(\alpha) \quad p,q \in Q_{\mathbf y}$,

\item[${{}}$]  $(\beta) \quad p \le^{\pr}_{\mathbf p,\kappa} q$.

\end{enumerate}
\item[${{}}$]   $(e) \quad$ for $q \in Q_{\mathbf y}$ let ap$_{\mathbf y}(q) = 
\ap_{\mathbf Q_{\mathbf y}}(q) = \{r \in Q_{\mathbf y}:
q \le^{\ap}_\kappa r$ and

\hskip25pt  supp$_\kappa(q,r) \subseteq \text{ supp}_\theta(p_{\alpha_{\mathbf
y}(q)},q)\}$.
\end{enumerate}
\end{definition}

\begin{observation}\label{2B.34}  
    Let $\mathbf y$ be a reasonable $\mathbf p$-parameter.
    
    0) If $p_1 \le p_2 \le q_2 \le q_1$ and $\kappa_1 \ge \kappa_2$ are from $\Theta$ \underline{then} {\rm supp}$_{\kappa_2}(p_2,q_2) \subseteq 
    \text{\rm supp}_{\kappa_1}(p_1,q_1)$.
    
    0A) If $p_1 \le p_2 \le p_3$ \underline{then} {\rm supp}$_{\kappa_1}(p_1,p_3) = 
    \text{\rm supp}_{\kappa_1}(p_1,p_2) \cup \text{\rm supp}_{\kappa_1}(p_2,p_3)$.
    
    1) For $q \in Q_{\mathbf y}$ the ordinal $\alpha_{\mathbf y}(q)$ is well
    defined $< \gamma_{\mathbf y}$.
    
    2) If $q_1 \le_{\mathbf y} q_2$ are from 
    $Q_{\mathbf y}$ \underline{then} $\alpha_{\mathbf y}(q_1) \le \alpha_{\mathbf y}(q_2)$.
    
    2A) If $q_1 \in Q_{\mathbf y}$ and $q_1 \le^{\ap}_{\mathbf p,\kappa} q_2$
    \then \, $q_2 \in Q_{\mathbf y},q_1 \le_{\mathbf y} q_2$ and  $\alpha_{\mathbf y}(q_1) = \alpha_{\mathbf y}(q_2)$.
    
    3) If $p \le^{\pr}_{\mathbf y} r$ and $q \in \ap_{\mathbf y}(p)$ \then \, $s := q \cup r$ belongs to $Q_{\mathbf y},s \in       \ap_{\mathbf y}(r)$ and $q \le^{\pr}_{\mathbf y} s$. 
\end{observation}

\begin{PROOF}{\ref{2B.34}}
0), 0A) Should be easy.

1) By the definitions of $q \in Q_{\mathbf y}$ and of $\alpha_{\mathbf y}(q)$.

2) For $\ell=1,2$ letting $\alpha_\ell = \alpha_{\mathbf y}(q_\ell)$ we
have $p_{\alpha_\ell} \le^{\ap}_\theta q_\ell \wedge 
\text{ supp}_\theta(p_{\alpha_\ell},q_\ell) \subseteq
u_{\alpha_\ell}$.   If $\alpha_2
   < \alpha_1$ then $p_{\alpha_2} \le p_{\alpha_1}
   \le^{\ap}_\theta q_1 \le q_2 \wedge p_{\alpha_2}
\le^{\ap}_\theta q_2$ \underline{hence} $p_{\alpha_2}
   \le^{\ap}_\theta q_1$ (by \ref{2B.20}(k)) and
   supp$_\theta(p_{\alpha_2},q_1) \subseteq 
\text{ supp}_\theta(p_{\alpha_2},q_2) \subseteq u_{\alpha_2}$ by the
definition of \ref{2B.33}(1) of supp, contradicting the choice of $\alpha_1$.

2A) We know $p_{\alpha_{\mathbf y}(q_1)} \le^{\ap}_{\mathbf p,\kappa} q_1$
by the definition of $\alpha_{\mathbf y}(q_1)$ but we assume $q_1
\le^{\ap}_{\mathbf p,\kappa} q_2$ and $\le^{\pr}_{\mathbf p,\kappa}$ is a
quasi order hence $p_{\alpha_{\mathbf y}(q_1)} \le^{\ap}_{\mathbf
p,\kappa} q$.  So by the definition $q_2 \in Q_{\mathbf y} \wedge
\alpha_{\mathbf y}(q_1) \ge \alpha_{\mathbf y}(q_2)$.  Also clearly $q_1
\le_{\mathbf p} q_2$ hence $q_1 \le_{\mathbf y} q_2$ hence by part (2),
$\alpha_{\mathbf y}(q_1) \le \alpha_{\mathbf y}(q_*)$, together we are
done.

3) Let $\kappa = \kappa_{\mathbf y}$ and $\theta = \theta_{\mathbf y},
p_\alpha = p^{\mathbf y}_\alpha$.
By Definition \ref{2B.33}(3)(e),(f) we know that $p
\le^{\pr}_{\mathbf p,\kappa} r$ and $p \le^{\ap}_{\mathbf
p,\kappa} q$.  By Claim \ref{2B.20}(f) we know that $s \in Q_{\mathbf
p}$ and $p \le^{\ap}_{\mathbf p,\kappa} q \le^{\pr}_{\mathbf
p,\kappa} s$ and $p \le^{\pr}_{\mathbf p,\kappa} r 
\le^{\ap}_{\mathbf p,\kappa} s$ recalling $s = q \cup r$, note

\begin{enumerate}
\item[$(*)_1$]  the ordinal $\beta := \alpha_{\mathbf y}(r) 
< \gamma_{\mathbf y}$ is well defined.
\end{enumerate}

[Why?  As $r \in Q_{\mathbf y}$.]

\begin{enumerate}
\item[$(*)_2$]   $\alpha_{\mathbf y}(s) = \alpha_{\mathbf y}(r) = \beta$.
\end{enumerate}

[Why?  As $p \in Q_{\mathbf y}$ the ordinal 
$\alpha := \alpha_{\mathbf y}(p) < \gamma_{\mathbf y}$ 
is well defined and by part (2) we have $\alpha \le \beta$.  So
clearly $p_\beta \le^{\ap}_{\mathbf p,\theta} r$ by the choice of
$\beta$ and $r \le^{\ap}_{\mathbf p,\kappa} s$ as said above,
hence by \ref{2.7}(e) recalling $\kappa < \theta$, we have
$\le^{\ap}_{\mathbf p,\kappa} \subseteq \le^{\ap}_\theta$ hence $r
\le^{\ap}_{\mathbf p,\theta} s$,
so together $p_\beta \le^{\ap}_{\mathbf p,\theta} s$.  Also $s=q \cup r$
hence $\supp_\theta(r,s) \subseteq \text{ supp}_\theta(p,q)$ and as
$q \in \ap_{\mathbf y}(p)$ necessarily $p 
\le^{\ap}_{\mathbf p,\kappa} q$ 
hence $p \le^{\ap}_{\mathbf p,\theta} q$ hence by part (2A)
$\supp_\theta(p,q) \subseteq \supp_\theta(p_\alpha,q) \subseteq
u^{\mathbf y}_{\alpha_{\mathbf y}(q)} = u^{\mathbf y}_{\alpha_{\mathbf y}(p)}
= u^{\mathbf y}_\alpha$ but $u^{\mathbf y}_\alpha \subseteq u^{\mathbf
y}_\beta$ as $\alpha \le \beta$.  Together supp$_\theta(r,s) \subseteq
u_\beta$, and by the choice of $\beta$ clearly supp$_\theta(p_\beta,r)
\subseteq u_\beta$ hence supp$_\theta(p_\beta,s) \subseteq \text{
supp}_\theta(p_\beta,r) \cup \text{ supp}_\theta(r,s) \subseteq
u_\beta \cup u_\beta = u_\beta$.  As we have shown earlier that
$p_\beta \le^{\ap}_{\mathbf p,\theta} s$ it follows that $s \in
Q_{\mathbf y}$ and $\alpha_{\mathbf y}(s) \le \beta$.  But $r \le_{\mathbf
p} s$ hence by part (2) we know that $\beta = \alpha_{\mathbf y}(r) \le
\alpha_{\mathbf y}(s)$ so necessarily $\alpha_{\mathbf y}(s) =
\alpha_{\mathbf y}(r) = \beta$, i.e. $(*)$ holds.]

So $p_{\alpha_{\mathbf y}(s)} \le^{\ap}_{\mathbf p,\theta}
s$ and supp$_\theta(p_{\alpha_{\mathbf y}(s)},s) = 
\text{ supp}_\theta(p_\beta,s) \subseteq u_\beta = u_{\alpha_{\mathbf
y}(s)}$ so together $s \in Q_{\mathbf y}$, the first statement in the
conclusion. 

Also $q \le^{\pr}_{\mathbf y} s$, for this check
$(e)(\alpha)+(\beta)$ of Definition \ref{2B.33}(3); for clause
$(\alpha)$: $q \in Q_{\mathbf y}$ is assumed, $s \in Q_{\mathbf y}$ was
just proved; for clause $(\beta)$ ``$q \le^{\pr}_{\mathbf
p,\kappa} s$" was proved in the beginning of the proof; so the third
statement in the conclusion holds.

Lastly, we check that $s \in \ap_{\mathbf y}(r)$, for this we
have to check the two demands in \ref{2B.33}(3)(f), now ``$s \in
Q_{\mathbf y}$" was proved above, ``$r \le^{\ap}_{\mathbf p,\kappa}
s$" was proved in the beginning of the proof and ``supp$_\kappa(r,s) \subseteq
\text{ supp}_\theta(p_{\alpha_{\mathbf y}(r)},s)$" holds as
supp$_\kappa(r,s) \subseteq \text{ supp}_\theta(r,s) \subseteq \text{
supp}_\theta(p_{\alpha_{\mathbf y}(r)},s) = \text{
supp}_\theta(p_\beta,s) = \text{ supp}_\theta(p_{\alpha_{\mathbf
y}(s)},s)$ is as required.   
\end{PROOF}

\begin{claim}
\label{2B.35}  
1) Assume $\kappa < \theta$ are successive members of $\Theta_{\mathbf p}$ and 
$(\forall \alpha < \partial_\theta)(|\alpha|^{< \partial_\kappa} 
< \partial_\theta)$ and $\mathbf y$ is a reasonable 
$\mathbf p$-parameter, $\kappa = \kappa_{\mathbf y}$ hence $\theta_{\mathbf
y} = \theta$ and $\bar p_{\mathbf y}$ is
   $\le^{\pr}_\theta$-increasing (hence also
$\le^{\pr}_\kappa$-increasing) and $\gamma_{\mathbf y}$ is a
   successor or a limit ordinal of cofinality $\ge \partial_\theta$.
\Then \, $\mathbf Q_{\mathbf y}$ is a $(\partial^+_\theta,\partial_\theta,<
   \partial_\theta)$-forcing.

2) If in addition $\gamma_{\mathbf y} = \alpha_* +1$ \then
\[
p_{\alpha_*} \Vdash ``\name {\mathbf G}_{\mathbf Q} \cap
Q_{\mathbf y} \text{ is a subset of } \bbQ_{\mathbf y} 
\text{ generic over } \mathbf V".
\]
\end{claim}

\begin{PROOF}{\ref{2B.35}}  
1) We should check for $\mathbf Q = \mathbf Q_{\mathbf y}$ (defined in
\ref{2B.33}) each of the clauses of Definition \ref{1c.15}.  Let
$p_\alpha = p^{\mathbf y}_\alpha,u_\alpha = u^{\mathbf p}_\alpha$.

\underline{Clause (a)}:  Trivial, just $\mathbf Q_{\mathbf y}$ 
has the right form, a quadruple.

\underline{Clause (b)}:  $(Q_{\mathbf y},\le_{\mathbf y})$ is a forcing notion.

Why?  By $\circledast_2(b)+(c)$ from \ref{2B.33}(3), i.e. $Q_{\mathbf
y}$ is a non-empty subset of $Q_{\mathbf p}$ because $\gamma_{\mathbf y} >
0$ so $p^{\mathbf y}_0 = p \in Q_{\mathbf y}$ and $\le_{\mathbf y}$ being
$\le_{\mathbf Q_{\mathbf p}} \rest Q_{\mathbf y}$ is a quasi order.

\underline{Clause (c)}:  $\le^{\pr}_{\mathbf y}$ is a quasi order on
$Q_{\mathbf y}$ and $p \le^{\pr}_{\mathbf y} q \Rightarrow p
\le_{\mathbf y} q \Rightarrow p \le_{\mathbf p} q$.

Why?  The first half holds because if $p_1 \le^{\pr}_{\mathbf y}
p_2 \le^{\pr}_{\mathbf y} p_3$ \then \,: we should check that
$p_1 \le^{\pr}_{\mathbf y} p_3$, i.e. clauses $(\alpha),(\beta)$
of $\circledast_2(e)$ of \ref{2B.33}(3) hold.  Now clause $(\alpha)$ is
obvious, for clause $(\beta)$ note $p_1 \le^{\pr}_{\mathbf p,\kappa} p_2 
\le^{\pr}_{\mathbf p,\kappa} p_3$ and
$\le^{\pr}_{\mathbf p,\kappa}$ is a partial order of $Q_{\mathbf
p}$, so $p_1 \le^{\pr}_{\mathbf p,\kappa} p_3$, and hence
$(\beta)$ holds.
 
The second part of clause (c) which says $p \le^{\pr}_{\mathbf y} q 
\Rightarrow p \le_{\mathbf y} q$ (recalling Claim
\ref{2B.20}(a)$(\beta)$) holds by the definition of
$\le_{\mathbf y},\le^{\pr}_{\mathbf y}$ in $\circledast_2(c),(e)$ 
of \ref{2B.33}(3).

\underline{Clause (d)$(\alpha)$}:  ap$_{\mathbf y}$ is a function with domain
$Q_{\mathbf y}$.

Why?  By $\circledast_2(f)$ of \ref{2B.33}(3).

\underline{Clause (d)$(\beta)$}:  if $q \in Q_{\mathbf y}$ then
$q \in \ap_{\mathbf y}(q) \subseteq Q_{\mathbf y}$.

Why?  By $\circledast_2(f)$ of \ref{2B.33}(3) trivially
ap$_{\mathbf y}(q) \subseteq Q_{\mathbf y}$.  Also we can check that $q
\in \ap_{\mathbf y}(q):q \in Q_{\mathbf y}$ by an assumption and
$q \le^{\ap}_\kappa q$ as $\le^{\ap}_\kappa$ is a quasi
order on $Q_{\mathbf p}$ and ``supp$_\kappa(q,q) \subseteq \text{
supp}_\theta(p_{\alpha_{\mathbf y}(q)},q)$" trivially because
supp$_\kappa(q,q) = \emptyset$.

\underline{Clause (d)$(\gamma)$}:  if $r \in \ap_{\mathbf y}(q)$ and $q
\in Q_{\mathbf y}$ then $r,q$ are compatible in $\bbQ_{\mathbf y}$.

Why?  As $r \in \ap_{\mathbf y}(q) \Rightarrow 
(q \le^{\ap}_\kappa r \wedge \{r,q\} \subseteq Q_{\mathbf y})
\Rightarrow q \le_{\mathbf y} r$.

\underline{Clause (d)$(\gamma)^+$}:  if $r \in \ap_{\mathbf y}(q)$ and $q
\le^{\pr}_{\mathbf y} q^+$ then $q^+,r$ are compatible in 
$(Q_{\mathbf y},\le_{\mathbf y})$, moreover there is $r^+ \in 
\ap_{\mathbf Q_{\mathbf y}}(q^+)$ such that $q^+ \Vdash_{\bbQ_{\mathbf y}}
``r^+ \in \name{\mathbf G}_{\bbQ_{\mathbf y}} \Rightarrow r \in 
\name{\mathbf G}_{\bbQ_{\mathbf y}}"$. 

This follows from \ref{2B.34}(3), by defining $s = r^+ = r \cup q^+$, 
which gives more.

\underline{Clause (e)}:  $(Q_{\mathbf y},\le^{\pr}_{\mathbf y})$ is 
$(< \partial_\theta)$-complete, recalling $\partial_\theta = \partial^\kappa$.

So assume $\langle q_\varepsilon:\varepsilon < \delta\rangle$ is
$\le^{\pr}_{\mathbf y}$-increasing and $\delta$ is a limit ordinal
$< \partial_\theta$; now $(Q_{\mathbf p},\le^{\pr}_\kappa)$ is $(<
\partial^\kappa)$-complete by Claim \ref{2B.24}(1) and $\langle
q_\varepsilon:\varepsilon < \delta\rangle$ is also
$\le^{\pr}_{\mathbf p,\kappa}$-increasing by clause
$\circledast_2(e)(\beta)$ of Definition \ref{2B.33}(3) hence $q_\delta :=
\cup\{q_\varepsilon:\varepsilon <\delta\}$ is a 
$\le^{\pr}_{\mathbf p,\kappa}$-lub of the sequence 
by \ref{2B.24}(1).  Now $\langle \alpha_\varepsilon := 
\alpha_{\mathbf y}(q_\varepsilon):\varepsilon <\delta\rangle$ is an 
$\le$-increasing sequence of ordinals $< \gamma_{\mathbf y}$ by Observation
\ref{2B.34}(2).  

Also by an assumption of \ref{2B.35}(1), the ordinal $\gamma_{\mathbf
y}$ is a successor ordinal or limit of cofinality $\ge
\partial_\theta$ but then $\delta < \text{ cf}
(\gamma_{\mathbf y})$.  So in both cases $\alpha_* = \sup\{\alpha_\varepsilon:
\varepsilon < \delta\}$ is an ordinal $< \gamma_{\mathbf y}$.  But $\bar
p^{\mathbf y}$ is $\le^{\pr}_{\mathbf p,\kappa}$-increasing continuous hence
$p_{\alpha_*} = \cup\{p_{\alpha_\varepsilon}:\varepsilon < \delta\}$
and similarly $u_{\alpha_*} = \cup\{u_{\alpha_\varepsilon}:\varepsilon
< \delta\}$.  Now easily $q_\delta$ is a
$\le^{\ap}_\theta$-extension of $p^{\mathbf y}_{\alpha_*}$, 
and supp$_\theta(p^{\mathbf y}_{\alpha_*},q_\delta) \subseteq 
\cup\{\text{supp}_\theta(p_{\alpha_{\mathbf y}(q_\varepsilon)},
q_\varepsilon):\varepsilon
< \delta\} \subseteq \cup\{u_{\alpha_\varepsilon}:\varepsilon <
\delta\} = u_{\alpha_*}$ which has cardinality $< \partial_\theta$ set each
hence $q_\delta \in Q_{\mathbf y}$.  Easily $q_\delta$ is as required.

\underline{Clause (f)}:  $(Q_{\mathbf y},\le^{\pr}_{\mathbf y})$ satisfies
the $\partial^+_\theta$-c.c.

Why?  Let $q_\varepsilon \in Q_{\mathbf y}$ for $\varepsilon <
\partial^+_\theta$, so $\alpha_\varepsilon := \alpha_{\mathbf
y}(q_\varepsilon)$ is well defined and without loss of generality $\langle
\alpha_\varepsilon:\varepsilon < \partial^+_\theta\rangle$ is constant
or increasing; also 
$p_{\alpha_\varepsilon} \le^{\ap}_\theta q_\varepsilon$ 
so by Definition \ref{2B.16} the set
supp$_\theta(p_{\alpha_{\mathbf y}(q_\varepsilon)},q_\varepsilon)$ 
has cardinality $< \partial_\theta$, so by the $\Delta$-system lemma, 
as in the proof of \ref{2B.24}(3) there are $\varepsilon(1) < \varepsilon(2) <
\partial^+_\theta$ such that:

\begin{enumerate}
\item[$(*)$]   if $i_1 \in \text{ supp}_\theta(p_{\alpha_{\varepsilon(1)}},
q_{\varepsilon(1)})$ and $i_2 \in \text{ supp}_\theta
(p_{\alpha_{\varepsilon(2)}},q_{\varepsilon(2)})$ then
\begin{enumerate}
\item[$(\alpha)$]    if $i_1 = i_2$ then $q_{\varepsilon(1)}(i) = 
q_{\varepsilon(2)}(i)$

\item[$(\beta)$]    if $i_1 E_\kappa i_2$ then $i_1,i_2 \in
\text{ supp}_\theta(p_{\alpha_{\varepsilon(1)}},q_{\varepsilon(1)})
\cap \text{ supp}_\theta(p_{\alpha_{\varepsilon(2)}},q_{\varepsilon(2)})$.
\end{enumerate}
\end{enumerate}

So $\varepsilon(1) < \varepsilon(2),\alpha_{\varepsilon(1)} \le
\alpha_{\varepsilon(2)},p_{\alpha_{\varepsilon(1)}}
\le^{\ap}_\theta q_{\varepsilon(1)},p_{\alpha_{\varepsilon(2)}}
\le^{\ap}_\theta q_{\varepsilon(2)}$.

Hence $q := q_{\varepsilon(1)} \cup q_{\varepsilon(2)}$ belongs to
$Q_{\mathbf p}$ is a $\le^{\ap}_\theta$-lub of
$\{q_{\varepsilon(1)},q_{\varepsilon(2)}\}$ and
$q_{\alpha_{\varepsilon(2)}} \le^{\ap}_\theta q$ hence $q \in
Q_{\mathbf y}$.  Also if $i \in \dom(q) \backslash \text{
Dom}(p_{\varepsilon(\ell)})$ then $i/E_\kappa$ is disjoint to
$\dom(p_{\varepsilon(\ell)})$ by $(*)(\beta)$; this implies
$p_{\varepsilon(\ell)} \le^{\pr}_\kappa q$ which means
$p_{\varepsilon(\ell)} \le^{\pr}_{\mathbf y}
q$ by \ref{2B.33}(3)(e), for $\ell=1,2$ so 
$q_{\varepsilon(1)},q_{\varepsilon(2)}$ are indeed commpatible in
$(Q_{\mathbf y},\le^{\pr}_{\mathbf y})$.

\underline{Clause (g)}:  if $\bar q = \langle q_\varepsilon:\varepsilon <
\partial_\theta\rangle$ is $\le^{\pr}_{\mathbf y}$-increasing,
\then \, for stationarily many limit $\zeta < \partial_\theta$ the sequence
$\bar q \restriction \zeta$ has an exact $\le^{\pr}_{\mathbf y}$-upper
bound (recalling that $\partial_\theta$ here stands for $\theta$ in
Definition \ref{1c.15}).

Why?  We prove more, that if cf$(\zeta) = \partial_\kappa$ and
$\langle q_\varepsilon:\varepsilon < \zeta\rangle$ is
$\le^{\pr}_{\mathbf y}$-increasing then the union $q =
\cup\{q_\varepsilon:\varepsilon < \zeta\}$ is an exact
$\le^{\pr}_{\mathbf y}$-upper bound. This suffices as
$\partial_\kappa < \partial_\theta$ and both are regular.
Now by \ref{2B.34}(2) the sequence $\langle \alpha_{\mathbf y}(q_\varepsilon):
\varepsilon < \zeta\rangle$ is $\le$-increasing
hence $\langle u_{\alpha_{\mathbf y}(q_\varepsilon)}:\varepsilon < \zeta\rangle$
is $\subseteq$-increasing and letting $\alpha_* =
\cup\{\alpha_{\mathbf y}(q_\varepsilon):\varepsilon < \zeta\}$
we have $\alpha_* < \gamma_{\mathbf y}$ as $\gamma_{\mathbf y}$ is a
successor ordinal or limit of cofinality $\ge \partial_\theta$; hence
$u_{\alpha_*} = \cup\{u_{\alpha_{\mathbf y}(q_\varepsilon)}:\varepsilon
<\zeta\}$, see \ref{2B.33}(2)(c).  

By the proof of 
clause (e) which we have proved above, clearly $q \in Q_{\mathbf y}$
and is a $\le^{\pr}_{\mathbf y}$-upper bound of $\langle
q_\varepsilon:\varepsilon < \zeta\rangle$.  But what about ``exact"?
we should check Definition \ref{1c.15}(2).  So assume $p \in \ap_{\mathbf
y}(q)$ and we should prove that for some $\varepsilon < \zeta$ and
$p' \in \ap_{\mathbf y}(q_\varepsilon)$ we have
$\Vdash_{\bbQ_{\mathbf y}}$ ``if $q,p' \in \name{\mathbf G}_{\bbQ_{\mathbf
y}}$ \then \, $p \in \name{\mathbf G}_{\bbQ_{\mathbf y}}$".

Note that $q \le^{\ap}_{\mathbf p,\kappa} p$ and supp$_\theta(q,p)
\subseteq u_{\alpha_*}$ by the definition of ap$_{\mathbf
y}(q)$, hence $u := \text{ supp}_\kappa(q,p)$ is a subset of
supp$_\theta(q,p) \subseteq u^{\mathbf y}_{\alpha_*}$ of cardinality $<
\partial_\kappa$.  As $\langle 
u^{\mathbf y}_{\alpha_\varepsilon}:\varepsilon < \zeta\rangle$ is
$\subseteq$-increasing with union $u^{\mathbf y}_{\alpha_*}$ necessarily
for some $\varepsilon < \zeta$ we have $u \subseteq
u_{\alpha_\varepsilon}$.  Let $p' = p \rest \dom
(p_\varepsilon)$, and check (as in earlier cases).

\underline{Clause (h)}:  if $\langle q_\varepsilon \colon \varepsilon <
\partial_\theta\rangle$ is $\le^{\pr}_{\mathbf y}$-increasing and
$r_\varepsilon \in \ap_{\mathbf y}(q_\varepsilon)$ for
$\varepsilon < \partial_\theta$ and $\xi < \partial_\theta$ 
\then \, for some $\zeta <
\partial_\theta$ we have $q_\zeta \Vdash_{\bbQ_{\mathbf y}}$ ``if $r_\zeta \in
\name{\mathbf G}_{\bbQ_{\mathbf y}}$ then $\xi \le
\text{ otp}\{\varepsilon < \zeta:p_\varepsilon \in 
\name{\mathbf G}_{\bbQ_{\mathbf y}}\}"$.  

This follows from \ref{2B.24}(3).

\underline{Clause (i)}:  ap$_{\mathbf y}(q)$ has cardinality $<
\partial_\theta$. 

Should be clear as $\alpha < \partial_\theta \Rightarrow |\alpha|^{<
\partial_\kappa} < \partial_\theta$ by an assumption of the claim and
$\alpha < \partial_\theta \Rightarrow |u_\alpha| <\partial_\theta$ 
(see \ref{2B.33}(3)(f)) and the definition of ap$_{\mathbf y}(q)$ in
$\circledast_2(e)$ of \ref{2B.33}(3).

Let $\alpha = \alpha_{\mathbf y}(q)$ so $\alpha <  \gamma_{\mathbf y}$ and
$|\ap_{\mathbf y}(q)| = |\{s:q \le^{\ap}_\kappa s$ and 
supp$_\kappa(q,s) \subseteq \text{ sup}_\theta(p_{\alpha_{\mathbf
y}(q)},q)\}| \le |\text{ supp}_\theta(p_{\alpha_{\mathbf y}(q)},q)|^{<
\kappa}$ but $|\text{supp}_\theta(p_{\alpha_{\mathbf y}(q)},q)| <
\partial_\theta$ and so by an assumption of the claim
$|\text{supp}_\theta(p_{\alpha_{\mathbf y}(q)},q)|^{< \kappa} <
\partial_\theta$ so we are done.

\underline{Clause (j)}:  Let $q_* \le_{\mathbf y} r$, so $\alpha \le \beta$
where $\alpha := \alpha_{\mathbf y}(q_*),\beta := \alpha_{\mathbf y}(r)$.  

By \ref{2B.20}(c) we can find a pair $(q,p)$ such that $q_*
\le^{\pr}_{\mathbf p,\kappa} q \le^{\ap}_{\mathbf p,\kappa}
r,q_* \le^{\ap}_{\mathbf p,\kappa} p \le^{\pr}_{\mathbf
p,\kappa} r,r = p \cup q$.  Now check.

2) Let $Q'' = \{p:p_{\alpha_*} \le^{\ap}_\theta p\}$.  So clearly
$Q'' \subseteq Q_{\mathbf y}$ and then $(\forall p \in Q_{\mathbf y})
(\exists q \in Q'')[p \le_{\mathbf y} q]$, by clause (f) of Claim 
\ref{2B.20}, i.e. $Q''$  is a dense subset of $Q_{\mathbf y}$ 
(by $\le_{\bbQ_{\mathbf y}} = \le_{\bbQ_{\mathbf p}} 
\restriction Q_{\mathbf y}$).  Really $q_1 \in Q'' \wedge q_1 
\le q_2 \in Q_{\mathbf y} \Rightarrow q_2 \in Q''$ by \ref{2B.34}(2).

Suppose ${\cI}$ is a dense open subset of $\bbQ_{\mathbf y}$ 
so ${\cI}_1 := {\cI} \cap Q''$ is dense in $\bbQ_{\mathbf y}$.

Let $\mathbf G$ be a subset of $\bbQ$ generic over $\mathbf V$ such that
$p_{\alpha_*}$ belongs to it. 
If ${\cI} \cap \mathbf G \ne \emptyset$ we are done, otherwise some
$q_1 \in \mathbf G$ is incompatible (in $\bbQ$) with every $q \in
{\cI}$.  As $\mathbf G$ is directed there is $q_2 \in \mathbf G$ such
that $p_{\alpha_*} \le q_2 \wedge q_1 \le q_2$.  As $p_{\alpha_*} \le
q_2$ by clause (c) of Claim \ref{2B.20} there is a 
$r_2 \in \bbQ$ such that 
$p_{\alpha_*} \le^{\ap}_\theta r_2 \le^{\pr}_\theta q_2$.
So $r_2 \in Q''$ hence by the assumption on ${\cI}$
there is $r_3 \in {\cI}$ such that $r_2 \le r_3$.  Now as $r_3 \in
{\cI}$ necessarily $p_{\alpha_*} \le^{\ap}_\theta r_3$ and
of course $p_{\alpha_*} \le r_2 \le r_3$ hence by clause (k) of Claim
\ref{2B.20} we have $r_2 \le^{\ap}_\theta r_3$.  Recalling $r_2
\le^{\pr}_\kappa q_2$ it follows by clause (f) of \ref{2B.20}
that there is $q_3 \in \mathbf Q$
such that $q_2 \le q_3 \wedge r_3 \le q_3$ hence $q_3 \Vdash
``\name{\mathbf G} \cap {\cI} \ne \emptyset"$ and
$q_1 \le q_3$, contradicting the choice of $q_1$. 
\end{PROOF}

\begin{claim}
\label{2B.36}  
If $\kappa \in \Theta \backslash
\{\mu\},\theta = \text{\rm min}(\Theta \backslash \kappa^+)$ and
$\theta = \mu \Rightarrow \partial_\theta < \mu$
and $(\forall \alpha < \partial_\theta)[|\alpha|^{< \partial_\kappa} <
\partial_\theta]$ and $\xi < \partial_\theta,\sigma < \partial_\theta$
\then \, $\Vdash_{\bbQ_{\mathbf p}} ``\partial^+_\theta \rightarrow
(\xi,(\xi;\xi)_\sigma)^2"$ . 

%https://www.overleaf.com/project/5df7869938109e0001f98360
\end{claim}

\begin{PROOF}{\ref{2B.36}}
Let $\sigma < \partial_\theta$ and $\xi <
\partial_\theta$ and we shall prove $\Vdash_{\bbQ_{\mathbf p}}
``\partial^+_\theta \rightarrow (\xi,(\xi;\xi)_\sigma)^2"$.  Toward
this assume $\name{\mathbf c}$ is a $\bbQ_{\mathbf p}$-name and $q^*
\in \bbQ_{\mathbf p}$ forces that $\name{\mathbf c}$ is a function
from $[\partial^+_\theta]^2$ to $1 + \sigma$.  Now we shall apply Claim
\ref{2B.24}(4) with $\theta$ here standing for $\kappa$ there.  
We choose $(p_i,u_i)$ by induction on $i < \partial^+_\theta$ such that:

\begin{enumerate}
\item[$\circledast_1$]   $(a) \quad p_i \in \bbQ_{\mathbf p}$ is
$\le^{\pr}_\theta$-increasing continuous with $i$ and $p_0 = q^*$ (so if $\theta = \mu,$ then $\bigwedge_{i} p_{i} = q^{\ast}$ because $\leq_{\theta}^{\pr}$ is equality), 

\item[${{}}$]   $(b) \quad$ for every $i < j < \partial^+_\theta$ the set
${\cI}_{i,j}$ is predense above $p_{j+1}$ where
\[
{\cI}_{i,j} = \{r:p_{j+1} \le^{\ap}_\theta r \text{ and } r
\text{ forces a value to } \name{\mathbf c}\{i,j\}\},
\]

\item[${{}}$]   $(c) \quad$ moreover ${\cI}_{i,j}$ has a subset
${\cI}'_{i,j}$ of cardinality $\le \partial_\theta$ which is

\hskip25pt predense over $p_{j+1}$,

\item[${{}}$]   $(d) \quad u_i$ is $\subseteq$-increasing continuous
and

\hskip25pt $u_i \subseteq \cup\{\alpha/E_\kappa:\alpha \in \text{
Dom}(p_i)\}$ and $|u_i| \le \partial_\theta$ for $i < \partial^+_\theta$,

\item[${{}}$]   $(e) \quad \alpha \in u_i \Rightarrow
(\alpha/E_\kappa) \subseteq u_i$,

\item[${{}}$]   $(f) \quad q \in {\cI}'_{i,j} 
\Rightarrow \text{\rm supp}_\kappa (p_{j+1},q) \subseteq u_{j+1}$.
\end{enumerate}

[Why is this possible?  For $i=0$ let $p_0 = q^*$, for $i$ limit let
$u_i = \cup\{u_j:j < i\}$ and $i < \partial^+_\theta$, and we like to
applly \ref{2B.24}(1) with $\kappa$ there standing for $\theta$ here,
so if $\partial^+_\theta \le \partial^\theta$ this is fine, otherwise
by \ref{2B.12}(2)(h) necessarily $\theta = \mu \wedge \partial_\theta
= \mu = 2^\theta$ contradicting an assumption.  Lastly, if $i = \iota +1$
then we have to deal with $\name{\mathbf c}\{\zeta,\iota\}$ for $\zeta <
\iota$, i.e. with $\le \partial_\theta$ names of ordinals $< \sigma$.
So we apply \ref{2B.24}(4) with $(p_\iota,\iota,\langle 
\name{\mathbf c}(j,\iota):j < \iota\rangle,\theta)$ here standing for $(p,A,\name
f,\kappa)$ there and get $p_i,\langle \cI_{j,\iota},\cI'_{j,\iota}:j <
\iota\rangle$ here standing for $q,\langle \cI_{q,\name f,a},
\cI'_{q,\name f,a}:a \in A \rangle$ there.  So the relevant parts of
clauses (a),(b),(c) hold.  Define $u_i$ as in clauses (d),(e),(f)
possible as $|\cI'_{j,\iota}| \le \partial_\theta,r \in \cI'_{j,\iota}
\Rightarrow |\text{supp}_\kappa(p_i,q)| \le \partial_\kappa <
\partial_\theta$.  So we are done carrying the induction.]

Let $\bar p = \langle p_i:i < \partial^+_\theta\rangle$ and $\bar u =
\langle u_i:i < \partial^+_\theta\rangle$.

So this will help to translate the problem from the forcing $\bbQ$ to 
the forcing $\bbQ_{\mathbf y}$.

We define $\mathbf y = (\kappa,\langle p_\alpha:\alpha <
\partial^+_\theta \rangle,\langle u_\alpha:\alpha <
\partial^+_\theta\rangle)$, so:

\begin{enumerate}
\item[$\circledast_2$]   $\mathbf y$ is a reasonable $\mathbf p$-parameter.
\end{enumerate}

[Why?  Check, see Definition \ref{2B.33}(2).]

\begin{enumerate}
\item[$\circledast_3$]   $\mathbf Q_{\mathbf y}$ is a
$(\partial^+_\theta,\partial_\theta,< \partial_\theta)$-forcing.
\end{enumerate}

[Why?  By Claim \ref{2B.35}(1).]

Now for $i<j<\partial^+_\theta$,

\begin{enumerate}
\item[$(*)$]   $(a) \quad {\cI}_{i,j}$ is predense in $\bbQ_{\mathbf y}$,

\item[${{}}$]   $(b) \quad$ if $q_1,q_2 \in {\cI}_{i,j}$ or just
$\in \bbQ_{\mathbf y}$, \then \, $q_1,q_2$ are compatible in 
$\bbQ_{\mathbf p}$ 

\hskip25pt iff they are compatible in $\bbQ_{\mathbf y}$.
\end{enumerate}

[Why?  The first clause (a) holds by our definitions.  For the second
clause (b), assume $q_1,q_2 \in \bbQ_{\mathbf y}$.  If they are
compatible in $\bbQ_{\mathbf y}$, then clearly they are compatible in
$\bbQ_{\mathbf p}$.  To show the other direction, let $q$ be $q_1
\cup q_2$.  If $q \in \mathbb Q_{\mathbf y}$ we are done, since $q_1,q_2
\le_{\mathbf y} q$.  So let us prove that $q \in \bbQ_{\mathbf y}$.
Denote $\alpha_1 = \alpha_{\mathbf y}(q_1),\alpha_2 = \alpha_{\mathbf
y}(q_2)$ and without loss of generality $\alpha_1 \le \alpha_2$.  So
$p_{\alpha_1} \le^{\ap}_\theta q_1,p_{\alpha_2}
\le^{\ap}_\theta q_2$ and also $p_{\alpha_1}
\le^{\pr}_\theta p_{\alpha_2}$, and it follows from
\ref{2B.20}(f)$(\delta)$ that $p_{\alpha_2} \le^{\ap}_\theta
q$.  Moreover, supp$_\theta(p_{\alpha_2},q) \subseteq \text{
supp}_\theta(p_{\alpha_1},q_1) \cup \text{
supp}_\theta(p_{\alpha_2},q_2) \subseteq u_{\alpha_1} \cup
u_{\alpha_2} = u_{\alpha_2}$.  Together, $q \in \bbQ_{\mathbf y}$ and
we are done.]

So we can define a $\bbQ_{\mathbf y}$-name $\name{\mathbf c}'$ 
as follows; for $q \in \bbQ_{\mathbf y}$ 

\[
q \Vdash_{\bbQ_{\mathbf y}} ``\name{\mathbf c}'\{i,j\} = t"
\text{ iff } q \Vdash_{\bbQ_{\mathbf p}} 
``\name{\mathbf c}\{i,j\} = t".
\]

So by $(*)$
\[
\Vdash_{\bbQ_{\mathbf y}} ``\name{\mathbf c}':
[\partial^+_\theta]^2 \rightarrow \sigma".
\]

Now by claim \ref{1c.25} for some $\bbQ_{\mathbf y}$-name and a sequence
$\langle \name \alpha_\varepsilon,\name \beta_\varepsilon:
\varepsilon < \xi\rangle$ we have

\begin{equation*}
\begin{array}{clcr}
\Vdash_{\bbQ_{\mathbf y}} ``&\text{the sequence }\langle 
\name\alpha_\varepsilon,\name \beta_\varepsilon:\varepsilon < \xi\rangle 
\text{ is as required in Definition \ref{0.3}} \\
  &(\text{for } \partial^+_\theta \rightarrow (\xi,(\xi;\xi)_\sigma)^2)".
\end{array}
\end{equation*}

So for each $\varepsilon < \xi$ there is a maximal antichain 
${\cJ}_\varepsilon$ of $\bbQ_{\mathbf y}$ of elements forcing a 
value to $(\name \alpha_\varepsilon,\name \beta_\varepsilon)$ by 
$\bbQ_{\mathbf y}$.

But $\bbQ_{\mathbf y}$ satisfies the $\partial^+_\theta$-c.c. 
so $|{\cJ}_\varepsilon| \le \partial_\theta$
 hence for some $\alpha_* < \partial^+_\theta$ we have:

\begin{enumerate}
    \item[$(*)$]   ${\cJ}_\varepsilon \subseteq
    \{q:(\exists \alpha \le \alpha_*)(p_\alpha \le^{\ap}_{\mathbf Q}
    q)\}$ for any $\varepsilon < \xi$. 
\end{enumerate}

Recall that (by \ref{2B.35})

\begin{enumerate}
\item[$(*)$]   $p_{\alpha_*} \Vdash ``\name{\mathbf G}_{\mathbf Q} 
\cap Q_{\mathbf y \restriction (\alpha_*+1)}$ is a subset
of $Q_{\mathbf y \restriction (\alpha_* +1)}$ generic over $\mathbf V$".
\end{enumerate}

so we are done.  
\end{PROOF}

\begin{remark}
\label{2B.41}  
1) We can replace the exponent 2 by $n \ge 2$, so
getting suitable polarized partition relations; we intend to continue
elsewhere.

2) For exact such results provable in ZFC see \cite{EHMR} and \cite{Sh:95}.
\end{remark}

%%%%%%%%%%%%%%%%%%%%%%%%%%%%%%%%%%%%%%%%%%%%%%%%%%%
% Section 3
\newpage

\section{Simultaneous Partition Relations and General topology}\label{3} 

Recall (to simplify results we define $\hL^+(X) > \lambda > \cf(\lambda)$ using an elaborate definition for regulars).

\begin{definition}\label{top}  
    Let $X$ be a topological space. We define: 
    
    \begin{enumerate}
        \item[$(a)$]   the \emph{density} of $X$ is $d(X) \coloneqq \min \{|S| \colon S
        \subseteq X$ and $S$ is dense in $X\},$
        
        \item[$(b)$]   the \emph{hereditary density }  of $X$ is:
        
        $\hd(X) = \sup\{\lambda:X$ has a subspace of density $\ge \lambda\},$
        
        \item[$(c)$]   $\hd^+(X) = \widehat{\hd}(X) = 
        \sup\{\lambda^+:X$ has a subspace of density $\ge \lambda\},$
        
        \item[$(d)$]   $X$ is not $\lambda$-Lindel\"of if there is a family
        $\{\cU_\alpha:\alpha < \lambda\}$ of open subsets of $X$ whose union is
        $X$ but $w \subseteq \lambda \wedge |w| < \lambda \Rightarrow
        \cup\{\cU_\alpha:\alpha \in w\} \ne X,$
        
        \item[$(e)$]   the hereditarily Lindel\"of number of $X$ is:
        
        $\hL(X) = \widehat{\text{hL}} (X) = \sup\{\lambda$ \,: there are $x_\alpha
        \in X$ and ${\cU}_\alpha \in \text{ open}(X)$ for $\alpha <
        \lambda$, such that $x_\alpha \in {\cU}_\alpha$ and $\alpha < \beta
        \Rightarrow x_\beta \notin {\cU}_\alpha\},$
        
        \item[$(f)$]  $\hL^+(X) = \sup\{\lambda^+$: there are $x_\alpha \in
        X,\cU_\alpha$ for $\alpha < \lambda$ as above$\},$
        
        \item[$(g)$]   the spread of $X$ is $s(X) = \sup\{\lambda:X$ has a
        discrete subset with $\lambda$ points$\}$,
        
        \item[(h)] $s^+(X) = \hat s(X) =
        \sup\{\lambda^+:X$ has a discrete subspace with $\lambda$ points$\}$.
    \end{enumerate}
\end{definition}

Our starting point was the following problem (\ref{0z.1}) of Juhasz-Shelah \cite{Sh:899}.

\begin{problem}  
    Assume $\aleph_1 < \lambda < 2^{\aleph_0}$.  Does there  exist a hereditarily Lindel\"of Hausdorff regular space of density $\lambda$?
    
    We answer negatively by a consistency result but then look again at related problems on hereditary density, Lindel\"ofness and spread; 
    our main theorem is \ref{1t.41} getting consistency for all
    cardinals.
    
    We also try to clarify the relationships of this and related  partition relations to $\chi \rightarrow [\theta]^2_{2  \kappa,2}$, recalling that by \cite{Sh:276}, consistently,
    e.g. $2^{\aleph_0} \rightarrow [\aleph_1]^2_{n,2}$ for $n < \omega$.  
    Now, see \ref{1t.33} below, $2^{\aleph_0} \rightarrow [\aleph_1]_{2n,2}$ 
    implies $2^{\aleph_0} \rightarrow
    (\aleph_1,(\aleph_1;\aleph_1)_n)^2$ and by
    \ref{1t.37} it implies $\gamma < \aleph_1 \Rightarrow 2^{\aleph_0}
    \rightarrow (\gamma)^2_n$, see on the consistency of this
    Baumgartner-Hajnal in \cite{BaHa73}, and Galvin in \cite{Gal75}. 
    
    On cardinal invariants in general topology, in particular, $s(X)$,{\rm hd}$(X)$,{\rm hL}$(X)$, see Juhasz \cite{Ju80}; in  particular recall the obvious.
\end{problem}

\begin{observation}
    For a Hausdorff topological space $X$:
    
    \begin{enumerate}
        \item[$(a)$]  $\hL(X) \ge s(X),$
        
        \item[$(b)$]  $\hd(X) \ge s(X),$
        
        \item[$(c)$]  for $\lambda$ regular, $X$ is hereditarily $\lambda$-Lindel\"of (i.e. every subspace is $\lambda$-Lindelof) iff  there are $x_\alpha \in X$ and  $\cU_\alpha$ for $\alpha < \lambda$ as in (e) of Definition \ref{top},
        
        \item[$(d)$]  we choose the second statement in (c) as the definition of ``$X$ is hereditarily $\lambda$-Lindel\"of" then \ref{1t.14}, \ref{1t.21} below hold also for $\lambda$ singular.
    \end{enumerate}
\end{observation}

\begin{conclusion}\label{3c.87} 
    Assume $\lambda = \lambda^{< \lambda} < 
    \mu = \mu^{< \mu}$ and GCH holds in $[\lambda,\mu]$, so 
    $\lambda \le \theta = \text{ cf}(\theta) \le
    \mu \Rightarrow \theta = \theta^{< \theta}$ and $\{\lambda,\mu\}
    \subseteq \Theta \subseteq \text{ Reg } \cap [\lambda,\mu]$ and  
    for $\theta \in \Theta$ we let $\partial_\theta = \theta$
     and let $\mathbf p =
    (\lambda,\mu,\Theta,\langle \partial_\theta:\theta \in \Theta\rangle)$.
    \Then \,
    
    \begin{enumerate}
        \item[$(a)$]   $\mathbf p$ is as required in Hypothesis \ref{2B.48},
        
        \item[$(b)$]  the forcing notion $\bbQ_{\mathbf p}$ satisfies:
        \begin{enumerate}
        \item[$(\alpha)$]   $\bbQ_{\mathbf p}$ is of cardinality $\mu$
        
        \item[$(\beta)$]   $\bbQ_{\mathbf p}$ is $(< \lambda)$-complete
        (hence no new sequence of length $< \lambda$ is added),
        
        \item[$(\gamma)$]   no cardinal is collapsed, no cofinality is
        changed,
        
        \item[$(\delta)$]   in $\mathbf V^{\bbQ_{\mathbf p}}$ we have
        $\lambda = \lambda^{< \lambda},2^\lambda = \mu$ and $\chi
        \notin[\lambda,\mu) \Rightarrow 2^\chi = (2^\chi)^{\mathbf V},$
        
        \item[$(\varepsilon)$]  if $\kappa < \theta$ are successive 
        members of $\Theta$ and $\theta$ is not a successor of singular (or
        just $\theta = \chi^+ \Rightarrow \chi^{< \kappa} = \chi$) \then \,
        $\lambda \rightarrow (\xi,(\xi;\xi)_\sigma)^2$ for any $\xi,\sigma <
        \theta$.
        \end{enumerate}
    \end{enumerate}
\end{conclusion}

\begin{PROOF}{\ref{3c.87}}
By \ref{2B.29} and \ref{2B.36}.
\end{PROOF}

The topological consequences from \ref{3c.87} in \ref{3c.89} hold by
\ref{1t.14} and \ref{1t.21} below, that is
\begin{conclusion}
\label{3c.89}  
We can add in \ref{3c.87} that 

\begin{enumerate}
\item[$(b)(\zeta)$]   if $\theta \in [\lambda,\mu) \cap \Theta$ 
is the successor
of the regular $\kappa$ \then \, for any Hausdorff regular topological
space $X$, we have $\hd(X) \ge \theta^+ \Rightarrow s^+(X) \ge \theta$
and also $\hL(X) \ge \theta^+ \Rightarrow s^+(X) \ge \theta$ so recalling
$\theta = \kappa^+$ we have $\hd(X) \ge \theta^+ \Rightarrow \text{
hL}(X) \ge s(X) \ge \kappa,\text{ hL}(X) \ge \theta^+ \Rightarrow
\text{ hd}(X) \ge s(X) \ge \kappa$

\item[$\,\,\,(\eta)$]   if $\theta \in (\lambda,\mu]$ is a limit
cardinal then $\hd(X) \ge \theta \vee \text{ hL}(X) \ge \theta
\Rightarrow s(X) \ge \theta$.
\end{enumerate}
\end{conclusion}

\begin{observation}
\label{1t.9} 
1) If $\lambda_1 \rightarrow (\xi_1;\xi_1)^2_{\kappa_1}$
and $\lambda_2 \ge \lambda_1,\xi_2 \le \xi_1,\kappa_2 \le \kappa_1$
\then \, $\lambda_2 \rightarrow (\xi_2;\xi_2)^2_{\kappa_2}$.

1A) Similarly for $\lambda \rightarrow (\xi,(\xi;\xi)_\kappa)^2$.

2) If $\lambda \rightarrow (\xi,(\xi;\xi)_\kappa)^2$ \then \,
$\lambda \rightarrow (\xi;\xi)^2_{1 +\kappa}$.

3) $\lambda \rightarrow (\xi + \xi;\xi + \xi)^2_\kappa$ 
implies $(\lambda,\lambda)
\rightarrow (\xi,\xi)^{1,1}_\kappa$; the polarized partition.
\end{observation}

\begin{claim}
\label{1t.14}
$X$ has a discrete subspace of size $\mu$, i.e. $s^+(X) > \mu$ 
(hence is not hereditarily $\mu$-Lindel\"of) \when \,:

\begin{enumerate}
\item[$(a)$]   $\lambda \rightarrow (\mu,(\mu;\mu))^2$

\item[$(b)$]   $X$ is a Hausdorff, moreover a regular $(=T_3)$ 
topological space

\item[$(c)$]   $X$ has a subspace of density $\ge \lambda$.
\end{enumerate}
\end{claim}

\begin{remark}
\label{top.1}  
The proofs of \ref{1t.14}, \ref{1t.21} are similar to older proofs.
\end{remark}

\begin{PROOF}{\ref{1t.14}}
$X$ has a subspace $Y$ with density $\ge \lambda$, by clause
(c) of the assumption.  We choose $x_\alpha,C_\alpha$ by induction
on $\alpha < \lambda$ such that

\begin{enumerate}
\item[$\circledast$]   $(\alpha) \quad x_\alpha \in Y,$

\item[${{}}$]   $(\beta) \quad C_\alpha =$ the closure of 
$\{x_\beta:\beta < \alpha\},$

\item[${{}}$]   $(\gamma) \quad x_\alpha \notin C_\alpha$.
\end{enumerate}

This is possible as $Y$ has density $\ge \lambda$.

Let $u^1_\alpha$ be an open neighborhood of $x_\alpha$ disjoint to
$C_\alpha$.

Let $u^2_\alpha$ be an open neighborhood of $x_\alpha$ whose closure,
$c  \ell(u^2_\alpha)$ is $\subseteq u^1_\alpha$.  Why does it exist?  As $X$
is a regular $(= T_3)$ space.

We define $\mathbf c:[\lambda]^2 \rightarrow \{0,1\}$ as follows: 

\begin{enumerate}
\item[$(*)$]   if $\alpha < \beta$ then $\mathbf c\{\alpha,\beta\} = 1$ 
iff $x_\beta \in u^2_\alpha$.
\end{enumerate}

By the assumption $\lambda \rightarrow (\mu,(\mu;\mu))^2$ at least
one of the following cases occurs.

\underline{Case 1}:  There is an increasing sequence $\langle
\alpha_\varepsilon:\varepsilon < \mu\rangle$ of ordinals $< \lambda$ 
such that $\varepsilon < \zeta < \mu 
\Rightarrow \mathbf c\{\alpha_\varepsilon,\alpha_\zeta\} = 0$.

This means that $\varepsilon < \zeta < \mu \Rightarrow
x_{\alpha_\zeta} \notin u^2_{\alpha_\varepsilon}$.  But if
$\varepsilon < \zeta < \mu$ then $u^2_{\alpha_\zeta}$ is an open
neighborhood of $x_{\alpha_\zeta}$ included in
$u^1_{\alpha_\zeta}$ which is disjoint to $C_{\alpha_\zeta}$ and
$x_{\alpha_\varepsilon} \in C_{\alpha_\zeta}$ so $x_{\alpha_\varepsilon}
\notin u^2_{\alpha_\zeta}$.

Lastly, $x_{\alpha_\varepsilon} \in u^2_{\alpha_\varepsilon}$ by the
choice of $u^2_{\alpha_\varepsilon}$.  Together we are done,
i.e.
$\langle(x_{\alpha_\varepsilon},u^2_{\alpha_\varepsilon}):\varepsilon <
\mu\rangle$ is as required.

\underline{Case 2}:  There is a sequence
$\langle(\alpha_\varepsilon,\beta_\varepsilon):\varepsilon <
\mu\rangle$ such that:

\begin{enumerate}
\item[$(*)_1$]   $\varepsilon < \zeta < \mu \Rightarrow
\alpha_\varepsilon < \beta_\varepsilon < \alpha_\zeta < \lambda.$

\item[$(*)_2$]   $\varepsilon < \zeta \Rightarrow \mathbf
c\{\alpha_\varepsilon,\beta_\zeta\} = 1$, really $\varepsilon \le \zeta$
suffice.
\end{enumerate}

So,

\begin{enumerate}
\item[$(*)_3$]   $\varepsilon < \zeta \Rightarrow x_{\beta_\zeta} \in
u^2_{\alpha_\varepsilon}$.
\end{enumerate}
but now for every $\varepsilon < \mu$ let

\begin{enumerate}
\item[$(*)_4$]   $y_\varepsilon := x_{\beta_{2 \varepsilon}}$ and
$u^3_\varepsilon := u^2_{\beta_{2 \varepsilon}}
\backslash c \ell(u^2_{\alpha_{2 \varepsilon +1}})$.
\end{enumerate}

So,

\begin{enumerate}
\item[$(a)$]   $u^3_\varepsilon = u^2_{\beta_{2 \varepsilon}} \backslash c
\ell(u^2_{\alpha_{2 \varepsilon +1}})$ is open (as open minus closed).

\item[$(b)$]   $y_\varepsilon \in u^3_\varepsilon$.
\end{enumerate}

[Why?  Recall $y_\varepsilon = x_{\beta_{2 \varepsilon}}$ belongs to
$u^2_{\beta_{2 \varepsilon}}$ (by the choice of
$u^2_{\beta_{2 \varepsilon}}$) and not to
$u^1_{\alpha_{2 \varepsilon +1}}$ (as $u^1_{\alpha_{2 \varepsilon +1}}$ 
is disjoint to $C_{\alpha_{2 \varepsilon+1}}$ while $x_{\beta_{2
\varepsilon}} \in C_{\alpha_{2 \varepsilon +1}})$ hence not 
to $c \ell(u^2_{\alpha_{2 \varepsilon +1}})$ being a subset 
of $u^1_{\alpha_{2 \varepsilon +1}}$.  Together $y_\varepsilon$
belongs to $u^2_{\beta_{2 \varepsilon}} \backslash 
c \ell(u^2_{\alpha_{2 \varepsilon+1}}) = u^3_\varepsilon$.]

\begin{enumerate} 
\item[$(c)$]   if $\varepsilon < \zeta < \mu$ then $y_\zeta \notin
u^3_\varepsilon$.
\end{enumerate}

[Why?  Now $y_\zeta = x_{\beta_{2 \zeta}}$ belongs to $u^2_{\alpha_{2
\varepsilon +1}}$ by $(*)_3$ as $2 \varepsilon + 1 < 2 \zeta$ which
follows from $\varepsilon < \zeta$ hence $y_\zeta$ belongs to 
$c\ell(u^2_{\alpha_{2 \varepsilon +1}})$ hence $y_\zeta \notin 
u^3_\varepsilon$ by the definition of $u^3_\varepsilon$.]

\begin{enumerate}
\item[$(d)$]   if $\zeta < \varepsilon < \mu$ then $y_\zeta \notin
u^3_\varepsilon$.
\end{enumerate}

[Why?  As $u^3_\varepsilon \subseteq u^2_{\beta_{2 \varepsilon}}$ and
the latter is disjoint to $C_{\beta_{2 \varepsilon}}$ to which
$x_{\beta_{2 \zeta}} = y_\zeta$ belongs.]

Together $\langle(y_\varepsilon,u^3_\varepsilon):\varepsilon <
\mu\rangle$ exemplifies that we are done.
\end{PROOF}

\begin{claim}
\label{1t.21}  
$X$ has a discrete subspace of size $\mu$ \when \,:

\begin{enumerate}
\item[$(a)$]   $\lambda \rightarrow (\mu,(\mu;\mu))^2,$

\item[$(b)$]  $X$ is a Hausdorff moreover a regular $(= T_3)$ topological space,

\item[$(c)$]   {\rm hL}$^+(X) > \lambda$, i.e. if $\lambda$ is a
regular cardinal this means that $X$ is not hereditarily
$\lambda$-Lindel\"of
\end{enumerate}
\end{claim}

\begin{PROOF}{\ref{1t.21}}
Similar to \ref{1t.14}.  We choose
$\langle(x_\alpha,u^1_\alpha):\alpha < \lambda\rangle$ such that
$u^1_\alpha$ is an open subset of $X,x_\alpha \in u^1_\alpha$ and
$u^1_\alpha \cap \{x_\beta:\beta \in (\alpha,\lambda)\} = \emptyset$.
We can choose them as $\hL^+(X) > \lambda$.
We then choose an open neighborhood $u^2_\alpha$ of $x_\alpha$ such
that $c \ell(u^2_\alpha) \subseteq u^1_\alpha$.  We then define $\mathbf
c:[\lambda]^2 \rightarrow \{0,1\}$ as follows

\begin{enumerate}
\item[$(*)$]   if $\alpha < \beta$ then $\mathbf c\{\alpha,\beta\}=1$
iff $x_\alpha \in u^2_\beta$.
\end{enumerate}

We continue as in the proof of \ref{1t.14}, but now, in Case 2

\begin{enumerate}
\item[$(*)'_3$]   $\varepsilon < \zeta \Rightarrow
x_{\alpha_\varepsilon} \in u^2_{\beta_\zeta}$
\end{enumerate}

and let

\begin{enumerate}
\item[$(*)'_4$]   $y_\varepsilon := x_{\alpha_{2
\varepsilon}},u^3_\varepsilon := u^2_{\alpha_{2 \varepsilon}}
\backslash c \ell(u^2_{\beta_{2 \varepsilon +1}})$.
\end{enumerate}
\end{PROOF}

Now we come to our main result.
\begin{theorem}
\label{1t.41}
\underline{The Main Theorem}.  

It is consistent (using no large cardinals)
that:

\begin{enumerate}
\item[$(*)$]   $(\alpha) \quad 2^\mu$ is $\mu^+$ if $\mu$ is strong
limit singular and always $2^\mu$

\hskip25pt   is the successor of a singular cardinal,

\item[${{}}$]   $(\beta) \quad$  for every $\mu$ we have 
$\mu \le \chi < 2^\mu \Rightarrow 2^\chi = 2^\mu,$

\item[${{}}$]   $(\gamma) \quad \text{\rm hd}(X) \ge \theta \Leftrightarrow
\text{\rm hL}(X) \ge \theta \Leftrightarrow \text{\rm s}(X) 
\ge \theta$ for any limit cardinal $\theta$

\hskip25pt  and Hausdorff regular $(=T_3)$ topological space $X,$ 

\item[${{}}$]   $(\delta) \quad \text{\rm hd}(X) \le \text{\rm s}
(X)^{+3}$ and \text{\rm hL}$(X) \le \text{\rm s}(X)^{+3}$ for 
any Hausdorff regular

\hskip25pt $(= T_3)$ topological space,

\item[${{}}$]   $(\varepsilon) \quad$ in $(\delta)$ we can replace
{\rm s}$(X)^{+3}$ by {\rm s}$(X)^{+2}$ except when $s(X)$ is regular,

\item[${{}}$]   $(\zeta) \quad$ in particular, if $X$ is a (Hausdorff
regular topological space which is) 

\hskip25pt  Lindel\"of or of countable density or just $s(X) 
= \aleph_0$ \then 

\hskip25pt {\rm hd}$(X) + \text{\rm hL}(X) \le \aleph_2,$

\item[${{}}$]   $(\eta) \quad$ if $X$ is a Hausdorff
space\footnote{is interesting because usually $2^\chi = 2^{(\chi^+)}$,
see clause $(\alpha)$} \then \, $|X| < 2^{(\text{\rm hd}(X)^+)}$

\item[${{}}$]   $(\theta) \quad$ if $X$ is a Hausdorff space
\then \, $w(X) \le 2^{(\text{\rm hL}(X)^+)}$

\item[${{}}$]   $(\iota) \quad$ if $2^\mu > \mu^+$ \then \, $\mu^{++}
\rightarrow (\xi,(\xi;\xi)_\mu)^2$ for $\xi < \mu^+$.
\end{enumerate}
\end{theorem}

\begin{remark}
\label{1t.42} 
In the Theorem \ref{1t.41} above:

1) If we use less sharp results in \S1,\S2,\S3 we should above just use
$(\hd(X))^{+n(*)}$ for large enough $n(*)$.

2) We may like to improve clause $(\eta)$ to 
$\le 2^{\hd(X)}$.  If below we
choose $\mu_{\varepsilon +1}$ strongly inaccessible (so we need to
 assume $\mathbf V \models$ ``there are unboundedly many strong
 inaccessible cardinals and clause $(\alpha)$ is changed"), 
nothing is lost, we have
$\lambda_{\varepsilon +1}  = \mu_{\varepsilon +1}$ then we can add:

\begin{enumerate}
\item[$(\eta)^+$]   for any Hausdorff space $X,|X| < 2^{\hd(X)}$
except (possibly) when $\hd(X)$ is strong limit singular.
\end{enumerate}

3) Similarly for clause $(\theta)$ about $w(X) \le 2^{\hL(X)}$.

4) Probably using large cardinal we can eliminate also the
exceptional case in $(\eta)^+$; it seemed that a similar situation
is the one in Cummings-Shelah \cite{Sh:541}, 
but we have not looked into this.

5) We may wonder whether in clause $(\zeta)$ we can replace $\aleph_2$
   by $\aleph_1$ and similarly for other cardinals, hopefully
see \cite{Sh:F884}.
\end{remark}

\begin{PROOF}{\ref{1t.41}}
We can assume $\mathbf V$ satisfies G.C.H.  We choose
$\langle(\lambda_\varepsilon,\mu_\varepsilon):\varepsilon$ an
ordinal$\rangle$ such that:

\begin{enumerate}
\item[$\circledast$]   $(a) \quad \lambda_0 = \mu_0 = \aleph_0$,

\item[${{}}$]   $(b) \quad \lambda_\varepsilon < \text{ cf}
(\mu_{\varepsilon +1}) < \mu_{\varepsilon +1},$ 

\item[${{}}$]   $(c) \quad \lambda_{\varepsilon +1}$ is the first
regular $\ge \mu_{\varepsilon +1}$,

\item[${{}}$]   $(d) \quad$  for limit $\varepsilon$ we have
$\lambda_\varepsilon$ is the first regular cardinal

\hskip25pt  $\ge \mu_\varepsilon := \cup\{\lambda_\zeta:\zeta < \varepsilon\}$.
\end{enumerate}

Now let $\mathbf p_\varepsilon= (\lambda_\varepsilon,\lambda_{\varepsilon +1},
\Theta_\varepsilon,\bar\partial_\varepsilon)$ where
$\Theta_\varepsilon,\bar\partial_\varepsilon$ are defined 
by $\Theta_\varepsilon =
 \text{ Reg } \cap [\lambda_\varepsilon,\lambda_{\varepsilon +1}],\bar
 \partial_\varepsilon = \langle \partial^\varepsilon_\theta:\theta \in
 \Theta_\varepsilon\rangle,\partial^\varepsilon_\theta = \theta$, so
 are chosen as in \ref{3c.87}.

So $\langle \mathbf p_\varepsilon:\varepsilon$ an ordinal$\rangle$ is a
class.  We define an Easton support iteration $\langle 
\bbP_\varepsilon,\name{\bbQ}_\varepsilon:\varepsilon
\in \text{ Ord}\rangle$ so $\cup\{\bbP_\varepsilon:
\varepsilon \in \text{ Ord}\}$ is a class forcing,
choosing the $\bbP_\varepsilon$-name $\name{\bbQ}_\varepsilon$ 
such that $\Vdash_{\bbP_\varepsilon} ``\name{\bbQ}_\varepsilon = 
\bbQ_{\mathbf p_\varepsilon}$, i.e. $\name{\bbQ}_\varepsilon$ is 
defined as in Definition \ref{2B.16} for the parameter $\mathbf
p_\varepsilon$ (in the universe $\mathbf V^{\bbP_\varepsilon}$ of course".

As in $\mathbf V^{\mathbb P_\varepsilon}$ section two is applicable for
$\mathbf p_\varepsilon$ so in $\mathbf V^{\bbP_{\varepsilon +1}}$, 
the conclusions of
\ref{3c.87}, \ref{3c.89} hold and $2^{\lambda_\varepsilon} =
\lambda_{\varepsilon +1}$ so cardinal arithmetic should be clear, in
particular, clause $(\alpha)$ holds.  Of course, forcing with 
$\bbP_\infty/\bbP_{\varepsilon +1}$ does not change those conclusions as
it is $\lambda_{\varepsilon  +1}$-complete.

In $\mathbf V^{\mathbb P_\infty}$ we have enough cases of $\theta^+
\rightarrow (\xi,(\xi;\xi))^2$, i.e. clause $(\gamma)$ 
by \ref{2B.36}.  So, first, if $\chi \ge \text{ s}(X)$
belongs to $[\lambda_\varepsilon,\mu_{\varepsilon +1})$ and is regular we have
$\chi^{+2} \rightarrow (\chi;(\chi;\chi))^2$ and 
{\rm hd}$(X)$, {\rm hL}$(X) \le \chi^{+2}$.  But if $s(X) \in
[\lambda_\varepsilon,\mu_{\varepsilon +1})$ then $s(X)^+ <
\mu_{\varepsilon +1}$ recalling $\mu_\varepsilon$ is singular hence
$\hd(X),\hL(X) \le s(X)^{+3} < \mu_{\varepsilon +1}$. 
 
Second, if $\chi = s(X)$ belongs to no such interval then 
$\chi^+ = \lambda_\varepsilon,\chi = \mu_\varepsilon > 
\text{ cf}(\mu_\varepsilon)$ for some $\varepsilon$ hence 
recalling $\lambda_\varepsilon =
\lambda^{< \lambda_\varepsilon}_\varepsilon = 2^\chi$ (in 
$\mathbf V^{\bbP_\infty}$) we have the conclusion.
 So clause $(\delta)$ follows hence also clauses $(\gamma),(\varepsilon)$.

Let us deal with clause $(\eta)$, let $\chi = \text{ hd}(X)$.  First, 
if $\chi \in
[\lambda_\varepsilon,\mu_{\varepsilon +1})$ we get $\hL(X) \le \chi^{+3} <
\mu_{\varepsilon +1}$ hence $|X| \le 2^{\chi^{+3}} = 2^\chi$ by the classical
inequality of de-Groot, ($|X| \le 2^{\text{hL}(X)}$; see
\cite{Ju80}).  Second, if $\chi$ belongs to no such interval,
then $\chi = \mu_\varepsilon \wedge \chi^+ =
\lambda_\varepsilon,2^{\mu_\varepsilon} = 2^\chi$ for some $\varepsilon$.
So $|X| \le 2^{2^{\text{hL}(X)}} \le 2^{2^\chi} = 2^{\chi^+}$ as required.

Clause $(\theta)$ is proved similarly.  
\end{PROOF}

\begin{theorem}
\label{1t.44}  
If in $\mathbf V$ there is a class of
(strongly) inaccessible cardinals, \then \, in some forcing extension

\begin{enumerate}
    \item[$(*)$]   $(\alpha) \quad 2^\mu$ is $\mu^+$ when $\mu$ is a
    strong limit singular cardinal and is a weakly inaccessible cardinal
    otherwise.
    
    \item[$(*)$]   $(\beta) - (\iota) \quad$ as in Theorem \ref{1t.41}.
\end{enumerate}
\end{theorem}

\begin{proof}  As in the proof of Theorem \ref{1t.41}.
\end{proof}

\begin{claim}
\label{1t.33}  
Assume $\chi \rightarrow [\theta]^2_{2 \kappa,2}$ where 
$\kappa \ge 2,\chi \le 2^\lambda$ and $\lambda
= \lambda^{< \lambda} < \theta = \text{\rm cf}(\theta)$.  \Then \, 
$\chi \rightarrow (\theta,(\theta;\theta)_\kappa)^2$.
\end{claim}

\begin{PROOF}{\ref{1t.33}}
    Let $\mathbf c:[\chi]^2 \rightarrow \kappa$ be given.
    
    Let $\eta_\alpha \in {}^\lambda 2$ for $\alpha < \chi$ be pairwise
    distinct.  We define $\mathbf d:[\chi]^2 \rightarrow 2 \kappa$ by: for
    $\alpha < \beta < \chi$ let $\mathbf d\{\alpha,\beta\}$ be $2\varepsilon
    + \ell$ when $\mathbf c\{\alpha,\beta\} = \varepsilon$ and $\ell = 1$
    iff $\ell \ne 0$ iff $\eta_\alpha <_{\text{lex}} \eta_\beta$
    (i.e. $\eta_\alpha(\ell g(\eta_\alpha \cap \eta_\beta)) <
    \eta_\beta(\ell g(\eta_\alpha \cap \eta_\beta))$.  As we are
    assuming $\chi \rightarrow [\theta]^2_{2 \kappa,2}$ there is ${\cU}
    \in [\chi]^\theta$ such that Rang$(\mathbf d \restriction [{\cU}]^2)$
    has $\le 2$ members, \wilog \, $\rm{otp}({\cU}) = \theta$.  If the
    number of members of Rang$(\mathbf d \restriction [{\cU}]^2)$
    is one we are done, so assume it is $\{2 \varepsilon_0 + \ell_0,2
    \varepsilon_1 + \ell_1\}$ where $\varepsilon_0,\varepsilon_1 < \kappa$ and
    $\ell_0,\ell_1 < 2$.  But we cannot have $\ell_1 = \ell_2$ by 
    the Sierpinski colouring properties as $\theta > \lambda$
    hence \wilog \, $\ell_0 = 0,\ell_1 = 1$.  If $\varepsilon_0 = \varepsilon_1
    = 0$ we are done, as then Case $(c)_0$ of Definition \ref{0z.4}(2)
    holds, so assume $\ell \in \{0,1\} \Rightarrow 
    \varepsilon_\ell \ne 0$.  Let $\Lambda = \{\eta \in
    {}^{\lambda >}2$ \,: for $\theta$ ordinals $\alpha \in {\cU}$ we have
    $\eta \triangleleft \eta_\alpha\}$.  Now $\Lambda$ has two
    $\triangleleft$-incomparable members (otherwise we get a contradiction by
    cf$(\theta) > \lambda$) say $\nu_0,\nu_1 \in \Lambda$ are
    $\triangleleft$-incomparable and \wilog \, $\nu_0 <_{\text{lex}} \nu_1$.
    
    So,
    
    \begin{enumerate}
        \item[$(*)$]   if $\nu_0 \trianglelefteq \eta_\alpha$ and $\nu_1
        \triangleleft \eta_\beta$ and $\alpha < \beta$ then $\mathbf
        c\{\alpha,\beta\} = \varepsilon_0$
        
        \item[$(*)$]   if $\nu_1 \triangleleft \eta_\alpha,\nu_0
        \triangleleft \eta_\alpha$ and $\alpha < \beta$ then $\mathbf
        c\{\alpha,\beta\} = \varepsilon_1$.
    \end{enumerate}
    
    As $\theta$ is regular and $\rm{otp}(\cU) = \theta$ we can choose
    $\alpha_\varepsilon,\beta_\varepsilon$ by induction on $\varepsilon <
    \theta$ such that:
    
    \begin{enumerate}
    \item[$\odot$]  $(a) \quad \alpha_\varepsilon \in \cU$ and
    $\alpha_\varepsilon > \text{ sup}\{\beta_\zeta:\zeta < \varepsilon\},$
    
    \item[${{}}$]  $(b) \quad \nu_0 < \eta_{\alpha_\varepsilon},$
    
    \item[${{}}$]  $(c) \quad \beta_\varepsilon \in \cU$ is 
    $> \alpha_\varepsilon,$
    
    \item[${{}}$]  $(d) \quad \nu_1 \triangleleft \eta_{\beta_\alpha}$.
    \end{enumerate}
    
    So Case $(c)_1$ of Definition \ref{0z.4}(2) holds.  So we are done.
\end{PROOF}

We can remark also

\begin{claim}\label{1t.37}  
    Assume $\lambda = \lambda^{< \lambda} <
    \text{\rm cf}(\theta)$ and $\chi \le 2^\lambda$ and $\chi \rightarrow
    [\theta]^2_{2 \kappa,2}$.  \Then \, for every ordinal $\gamma <
    \lambda^+$ we have $\chi \rightarrow (\gamma)^2_\kappa$.
\end{claim}

\begin{PROOF}{\ref{1t.37}}  
Without loss of generality $\kappa \ge 2$.

So let $\mathbf c:[\chi]^2 \rightarrow \kappa$.  Choose
$\langle \eta_\alpha:\alpha < \chi\rangle$ and $\mathbf d$ as in the
proof of \ref{1t.33} and let ${\cU} \subseteq \chi$ of order type
$\theta$ and $\{2 \varepsilon_0,2 \varepsilon_1+1\}$ be as there so
$\varepsilon_0,\varepsilon_1 < \kappa$.

As $\{\eta_\alpha:\alpha \in {\cU}\}$ is a subset of ${}^{\lambda
>}2$ of cardinality $\theta  > \lambda = \lambda^{< \lambda}$ clearly
(e.g. prove by induction on $\gamma < \lambda^+$ that) 
for every such ${\cU}$ there is ${\cU}' \subseteq {\cU}$ 
of order type $\gamma$ such that $\langle \eta_\alpha:\alpha \in
{\cU}'\rangle$ is $<_{\text{lex}}$-increasing.  So ${\cU}'$ is
as required, i.e. $\mathbf c \restriction [\{\eta_\alpha:\alpha \in
{\cU}'\}]^2$ is constantly $\varepsilon_0$ (of course also
$\varepsilon_1$ is O.K. if we use $<_{\text{lex}}$-decreasing
sequence). 
\end{PROOF}

\begin{remark}
If we use versions of $\chi \rightarrow [\theta]^2_{\kappa,2}$
with privilege positions for the value $0$, we can get corresponding
better results in \ref{1t.33}, \ref{1t.37}.
\end{remark}
\newpage

 %  PRIVATE PART 1, FINAL 

\bibliographystyle{amsalpha}
\bibliography{shlhetal}

\end{document}